\documentclass[11pt, a4paper, reqno]{amsart}

\usepackage[dvipsnames]{xcolor}
\usepackage[utf8]{inputenc}
\usepackage[english]{babel}
\usepackage{amsmath}        
\usepackage{amsfonts}       
\usepackage{amsthm}         
\usepackage{verbatim}
\usepackage{xargs}
\usepackage{caption}
\usepackage{subcaption}

\usepackage[colorlinks]{hyperref}
\usepackage[nameinlink]{cleveref}
\usepackage{tikz-cd}
\usetikzlibrary{decorations,calligraphy}
\usepackage{mathtools}

\usepackage{amssymb}
\usepackage{enumerate}
\usepackage{enumitem}
\usepackage{makecell}
\usepackage{pgfplots}
\usepackage{tabularx}

\usepackage[colorinlistoftodos]{todonotes}
\usepackage{xifthen}
\newcommand{\ifem}[3]{\ifthenelse{\isempty{#1}}{#2}{#3}}  
\newcommandx{\marie}[3][1=, 2=]{
	\ifem{#1}{
		\todo[color=magenta!50]{#3}
	}{
	\ifem{#2}{
		\todo[color=magenta!50, #1, caption={}]{#3}
	}{
		\todo[color=magenta!50, #1, caption={#2}]{ #3}
}}}

\newcommandx{\roland}[3][1=, 2=]{
	\ifem{#1}{
		\todo[color=green!50]{#3}
	}{
	\ifem{#2}{
		\todo[color=green!50, #1, caption={}]{#3}
	}{
		\todo[color=green!50, #1, caption={#2}]{ #3}
}}}

\usepackage{soul}
\usepackage{geometry}
\hypersetup{linkcolor=Cerulean, citecolor=Orange, urlcolor=DarkOrchid}

\usepackage[backend=bibtex,bibencoding=utf8,style=alphabetic]{biblatex}
\AtEveryBibitem{%
	\ifentrytype{article}{
		\clearfield{publisher}%
		\clearfield{issn}%
		\clearfield{isbn}%
		\clearfield{month}
	}{}
	\ifentrytype{inproceedings}{
		\clearfield{publisher}%
		\clearfield{issn}%
		\clearfield{isbn}%
	}{}
}
\AtBeginBibliography{\small}
\usepackage{xurl}

\theoremstyle{plain}
\newtheorem{thm}{Theorem}[subsection]
\newtheorem{manualtheoreminner}{}
\newenvironment{manualtheorem}[1]{%
	\renewcommand\themanualtheoreminner{#1}%
	\manualtheoreminner
}{\endmanualtheoreminner}

\theoremstyle{definition}
\newtheorem{defn}[thm]{Definition}
\newtheorem{proposition}[thm]{Proposition}
\newtheorem{corollary}[thm]{Corollary}
\newtheorem{lemma}[thm]{Lemma}
\newtheorem{rem}[thm]{Remark}
\newtheorem{example}[thm]{Example}

\newcommand{\im}{\operatorname{Im}}

\numberwithin{equation}{section}

\newcommand{\Z}{\mathbb Z}
\newcommand{\N}{\mathbb N}
\newcommand{\R}{\mathbb R}

\newcommand{\moment}[1][d]{\nu_{#1}}
\newcommand{\poly}{P_\xi}
\newcommand{\polyc}{P_\sigma}
\newcommand{\h}{\mathfrak{h}}
\newcommand{\point}[1]{\frac{\moment(#1)}{\inner{\xi,\moment(#1)}}}
\DeclareMathOperator{\ins}{ins}
\newcommand{\cT}{\mathring{T}}
\newcommand{\cI}{\mathring{I}}
\newcommand{\cS}{\mathring{S}}
\newcommand{\cD}{\mathring{D}}

\newcommand{\cfacets}[1][\cD]{\mathcal C_{#1}}

\usepackage[bb=boondox]{mathalfa}

\newcommand{\inner}[1]{\langle #1 \rangle}
\newcommand{\ma}{\begin{pmatrix} }
\newcommand{\trix}{\end{pmatrix} }
\newcommand{\sma}{ \left( \begin{smallmatrix} }
\newcommand{\strix}{ \end{smallmatrix} \right) }

\DeclareMathOperator{\conv}{conv}

\DeclareMathOperator{\sgn}{sgn}
\DeclareMathOperator{\cone}{cone}
\DeclareMathOperator{\interior}{int}

\DeclareMathOperator{\parity}{par}

\usepackage{xpatch}
\makeatletter   
\xpatchcmd{\@tocline}
{\hfil\hbox to\@pnumwidth{\@tocpagenum{#7}}\par}
{\ifnum#1<0\hfill\else\dotfill\fi\hbox to\@pnumwidth{\@tocpagenum{#7}}\par}
{}{}
\makeatother    

\makeatletter
 \def\l@subsection{\@tocline{2}{0pt}{1.15cm}{1cm}{}}
\def\l@subsubsection{\@tocline{3}{0pt}{8pc}{8pc}{}}
 \makeatother

\newcommand{\nocontentsline}[3]{}
\newcommand{\tocless}[2]{\bgroup\let\addcontentsline=\nocontentsline#1{#2}\egroup}

\makeatletter
\providecommand\@dotsep{5}
\def\listtodoname{List of Todos}
\def\listoftodos{\@starttoc{tdo}\listtodoname}
\makeatother

\date{}

\title[Veronese polytopes]{Veronese polytopes \\ \scriptsize{Extending the framework of cyclic polytopes}}

\author{Marie-Charlotte Brandenburg \and Roland Púček}
\address{Marie-Charlotte Brandenburg, KTH Royal Institute of Technology, Lindstedtsvägen 25, 114 28 Stockholm, Sweden}
\email{\href{mailto:brandenburg@kth.se}{brandenburg@kth.se}}
\address{Roland Púček, Mathematical Institute, University of Jena, Ernst-Abbe-Platz 1-2, 07743 Jena, Germany}
\email{\href{mailto:roland.pucek@uni-jena.de}{roland.pucek@uni-jena.de}}

\makeatletter
\@namedef{subjclassname@2020}{\textup{2020} Mathematics Subject Classification}
\makeatother

\apptocmd{\thebibliography}{\fontsize{11}{15}\selectfont}{}{}

\usepackage{soul}
\renewcommand{\emph}{\textit}
\sethlcolor{magenta!10}

\newcommand{\blockcomment}[1]{%
}%

\begin{document}

\maketitle

\vspace{-1.5em}

\begin{abstract}
This article introduces the theory of Veronese polytopes, a broad generalisation of cyclic polytopes.
These arise as convex hulls of points on curves with one or more connected components, obtained as the image of the rational normal curve in affine charts.
We describe their facial structure by extending Gale’s evenness condition, and provide a further combinatorial characterisation of facets via $\sigma$-parity alternating sequences.
Notably, we establish a bijective correspondence between combinatorial types of Veronese polytopes and partitions of finite sets equipped with a cyclic order, called circular compositions.
We show  that, although the only Veronese $3$-polytopes are the cyclic $3$-polytopes and the octahedron, in general dimension they form a rich and diverse class including all combinatorial types of simplicial $d$-polytopes with at most $d+3$ vertices, the cross-polytope and particular stacked polytopes.
In addition, we characterise which curves defining Veronese polytopes are $d$-order curves, and provide a closed formula for the number of facets of any Veronese polytope. 
\end{abstract}

\thispagestyle{empty}

\vspace{1em}

This article develops the theory of \emph{Veronese polytopes}, a natural and novel generalisation of cyclic polytopes. It emerges from the study of factorization structures in differential geometry \cite{apostolov2015ambitoric,pucek2022extremal,pucek2023factorization}.
In this context, Veronese polytopes arise as the class of polytopes compatible with the \emph{Veronese factorization structure}, one of the simplest non-trivial factorization structures. 
They form a large and diverse family of simplicial polytopes whose examples include all combinatorial types of $d$-dimensional simplicial polytopes on $d+1, d+2$ and $d+3$ vertices, cyclic polytopes, cross-polytopes, certain stacked polytopes, and a plenitude of unnamed simplicial polytopes.
Importantly, the aforementioned compatibility equips Veronese polytopes with a common framework which allows us to handle this set of seemingly unrelated polytopes at once by uniform methods.
This work presents the first systematic study of polytopes compatible with any factorization structure, opening new directions for research in polytope theory, with numerous applications to differential geometry. \\

The class of Veronese polytopes contains and exceeds the most classical of all polytopes constructed on curves:
The \emph{standard cyclic polytope} is the convex hull of finitely many points $\moment(t), t\in T$, on the \emph{moment curve} $\moment: \R \to \R^{d+1}, t \mapsto (1,t,t^2,\dots,t^d)$, and any polytope combinatorially equivalent to the standard cyclic polytope is a \emph{cyclic polytope}.
A choice of a non-zero $\xi \in (\R^{d+1})^* \setminus \bigcup_{t \in T} (\nu_d(t))^\circ$ gives a \emph{Veronese polytope} 
$$
P_\xi(T) = \conv\left(\point{t} \mid t \in T\right),
$$
where $(\nu_d(t))^\circ$ denotes the annihilator of $\nu_d(t)$ in $(\mathbb{R}^{d+1})^*$.
In particular, $\xi = (1,0,\dots,0)$ recovers the standard cyclic polytope.
Defined more abstractly, a Veronese polytope is the convex hull of finitely many points lying on the image of the \emph{factorization curve} of the Veronese factorization structure in an affine chart $\xi$. 
Notably, the factorization curve governs the entire geometry and combinatorics of Veronese polytopes.\\

We draw our motivation for this article from two directions: the theory of factorization structures in differential geometry, and the generalisation of cyclic polytopes as studied in discrete geometry and geometric combinatorics.

In differential geometry, one of the main research directions is to seek canonical geometric structures, often arising as extremal points of a (energy) functional, such as the heavily studied extremal Kähler metrics \cite{calabi1982extremal, calabi1985extremal}. Finding non-trivial explicit examples of these metrics is a challenging task, and several were provided ad hoc using toric geometry \cite{
	simanca1991kahler,
	simanca1992note,
	abreu1998kahler,
	article,
	apostolov2003geometry,
	apostolov2004hamiltonian,
	jsg/1310388900,
	apostolov2016ambitoric,
	apostolov2021cr}
Factorization structures offer a unifying framework encompassing all known explicit extremal toric Kähler metrics and providing new examples \cite{pucek2022extremal}.
Additionally, it determines extensive families of explicit toric Sasaki and Kähler geometries amenable to computations.
The connection to this work arises through the momentum map of a toric geometry whose image is a (rational) Delzant polytope.
When such a toric geometry is determined by a factorization structure, the Delzant polytope is dual to a compatible polytope, as defined in this article.
For the geometry to be compact, an often required attribute, one undergoes an involved process of compactification: the geometry is defined on the interior of a polytope and needs to be extended to the boundary. This process necessitates a detailed understanding of compatible polytopes.

In discrete geometry, our motivation lies in extending techniques of cyclic polytopes to Veronese polytopes, as outlined in the following.
To better position this article within existing literature, we recall that Veronese and cyclic polytopes share a common definition as convex hulls of points on curves, and explore the relationship between these curves and the resulting polytopes.
For cyclic polytopes, \cite{cordovil00_cyclicpolytopesoriented} finds that the convex hull of any finite number of points on a curve $\alpha: I\subset\mathbb{R} \to \mathbb{R}^d$ is a cyclic polytope if and only if $\alpha$ is a \textit{$d$-order curve}, i.e., a continuous injective map such that every affine hyperplane in $\mathbb{R}^d$ intersects $\im\alpha$ in at most $d$ points.
This elegant characterisation can be seen as a culmination of a considerable attention devoted to the theory of $d$-order curves, see e.g., \cite{caratheodory1911variabilitatsbereich,derry56_convexhullssimple,fabriciusbjerre62_polygonsordern,gale1963neighborly,sturmfels87_cyclicpolytopesd}.
To address the defining curve of a Veronese polytope, $t \mapsto \nu_d(t)/\langle \xi, \nu_d(t) \rangle$, note that its projectivisation is the same as that of the moment curve -- the rational normal curve, a distinguished curve of algebraic geometry.
This implies that all Veronese polytopes are simplicial (\Cref{prop:simplicial}), any simplicial polytope with at most $d+3$ vertices can be realized as a Veronese polytope (\Cref{thm:comb-types-few-vertices}), and that curves $\nu_d$ and $\nu_d/\langle \xi, \nu_d \rangle$ satisfy the hyperplane condition of $d$-order curves. However, we emphasise that Veronese polytopes stand out from the existing literature as their defining curves are often not connected, and therefore not $d$-order curves.

\begingroup
\hypersetup{hidelinks}
\begin{manualtheorem}{\Cref{d order curve}}
	\hypertarget{th:intro}{The}
	factorization curve $\psi$ of the Veronese factorization structure 
	in the affine chart $\xi$, i.e., the curve 
	$\nu_d/\langle \xi, \nu_d \rangle$,
	is a $d$-order curve if and only if $\xi$ lies on the rational normal curve \eqref{xi curve}.
\end{manualtheorem}
\endgroup

The study of cyclic polytopes dates back to the seminal works of Carathéodory \cite{caratheodory1911variabilitatsbereich} and Gale \cite{gale1963neighborly}.
A notable result is the description of their facial structure in terms of the non-negativity of certain univariate polynomials, known as \emph{Gale's evenness condition}.
We provide a generalisation of this condition for Veronese polytopes.

\enlargethispage{\baselineskip}

\begingroup
\hypersetup{hidelinks}
\begin{manualtheorem}{\Cref{geometric Gale}}
	Let $P_\xi(T)$ be a $d$-dimensional Veronese polytope and $S \subset T$ be of cardinality $d$. The unique hyperplane determined by the points
	\[
	\frac{\moment(t)}{\sum_{i=0}^d \xi_i t^i} \ , \quad t \in S,
	\]
	is a facet-supporting hyperplane of $P_\xi(T)$ if and only if the values
	\begin{gather}\label{eq:intro}
		\frac{\prod_{s \in S} (t - s)}{\sum_{i=0}^{d} \xi_i t^i},
		\hspace{.2cm}
		t \in T \backslash S,
	\end{gather}
	have equal signs.
\end{manualtheorem}
\endgroup

The $k$ roots of the polynomial $q_\xi(t) = \sum_{i=0}^d \xi_it^i$ from the denominator of \eqref{eq:intro} induce a partition of $T$ into $k+1$ discrete intervals, which, as we show, depends only on the chamber $\sigma$ of the \emph{hyperplane arrangement} $\mathcal H(T) = (\R^{d+1})^* \setminus \bigcup_{t \in T} (\nu_d(t))^\circ$ where $\xi$ belongs to.
This partition allows us to characterise facets of Veronese polytopes in a purely combinatorially fashion, via \emph{$\sigma$-parity alternating subsequences} of $\{1,\dots,|T|\}$ (\Cref{facets Veronese}), and as subsets of $T$ which admit a particular decomposition (\Cref{lem:facet-roots}). \\

The theory of factorization structures suggests an alternative, projective point of view on the construction of Veronese polytopes.
By projectivising the generating points of~$P_\xi(T)$, we obtain projective points $p_1,\ldots,p_{|T|}$ lying on the image of the rational normal curve $\psi: \mathbb{P}^1 \to \mathbb{P}^d$. As $\psi$ is injective and $\mathbb{P}^1 \cong \mathbb{S}^1$, we may interpret the preimages of the points $p_1,\dots,p_{|T|}$ as $|T|$ points on $\mathbb S^1$, yielding a set $\cT$ of cardinality $|T|$ equipped with a \emph{cyclic order} \cite{huntington16_setindependentpostulates,huntington35_interrelationsfour,coxeter93_realprojectiveplane}.
To capture the information carried originally by $\xi$ in this projective setting, we work with the \emph{points at infinity with respect to $\xi$}, namely,  $\im \psi \cap \xi^0$. The latter is the intersection of a degree $d$ curve with a hyperplane, and thus consists of at most $d$ elements. On the circle $\mathbb{S}^1$, we interpret these elements as separators, partitioning $\mathbb{S}^1$ (and thus $\cT$) into $k$ arcs, obtaining a \emph{circular composition} of $\cT$. The circular composition uniquely determines the combinatorics of the Veronese polytope.

\begingroup
\hypersetup{hidelinks}
\begin{manualtheorem}{\Cref{th:circular-compositions-comb-types}}
	\hypertarget{th:intro2}{Two} Veronese polytopes are combinatorially equivalent if and only if their circular compositions are isomorphic.
\end{manualtheorem}
\endgroup

The concept of a circular composition proves to be a powerful one as it easily identifies abundance of isomorphisms between Veronese polytopes, and carries complete and directly accessible information about vertices and facets (\Cref{th:circular-compositions-comb-types}, \Cref{vertices}).
Among others, it allows us to realise cross-polytopes as Veronese polytopes (\Cref{th:cross}), to prove that all $3$-dimensional Veronese polytopes are either cyclic or the octahedron (\Cref{cor:3d})
and provides examples of cyclic Veronese polytopes whose defining curves fail to be $d$-order curves (\Cref{prop:xi-unit-direction-combinatorial}).
Our findings are supported by computational evidence as displayed in  \Cref{table:small-dimensions,table:more-comb-types} (on \cpageref{table:small-dimensions,table:more-comb-types}), which show the number of combinatorial types of Veronese polytopes in small dimensions, in comparison to neighbourly and general simplicial polytopes (see also \Cref{rem:computational}).

Circular compositions also enable us to provide an explicit formula for the number of facets for Veronese polytopes. (\Cref{th:facet-count}).
Cyclic polytopes, being prototypical examples of \emph{neighbourly polytopes} \cite{gale1963neighborly,grunbaum1967convex,mcmullen1971convex}, maximise the number of facets among all simplicial $d$-polytopes by the \emph{Upper Bound Theorem} \cite{mcmullen70_maximumnumbersfaces}.
Due to their construction, it may be tempting to speculate that Veronese polytopes are neighbourly~as well. However, as we show in \Cref{th:stacked}, in every dimension and for every number of vertices the class of Veronese polytopes contains particular \emph{stacked polytopes}, the minimisers of the number of facets as established by the \emph{Lower Bound Theorem} \cite{barnette71_minimumnumbervertices,barnette73_prooflowerbound}. \\

Veronese polytopes are one of several directions for generalising cyclic polytopes. The literature presents various other generalisations, including \emph{almost cyclic polytopes}, a subclass of neighbourly polytopes \cite{shemer87_almostcyclicpolytopes}, and \emph{alternating uniform matroids}, which serve as an oriented matroid analogue of cyclic polytopes \cite{cordovil00_cyclicpolytopesoriented}.
Another direction is the study of convex hulls of infinitely many points on a curve.
Recently, questions from stochastic analysis and computational geometry motivated the study of convex hulls of a bounded segment of the moment curve and of limits of $d$-order curves \cite{pont2023newproofdescriptionconvex,amendola23_convexhullsofcurves}.
Similarly, convex hulls of monomial curves and general $1$-dimensional semi-algebraic sets are of interest in convex optimisation \cite{averkov24_convexhullsmonomial,scheiderer24_convexhullscurves}.

\tocless{\subsection*{Acknowledgements}}
MB was supported by the Wallenberg AI, Autonomous Systems and Software Program (WASP) funded by the Knut and Alice Wallenberg Foundation.

\begin{table}[h!]
	\begin{minipage}{1.03\textwidth}
	\begin{tabular}{|c|c|ccc|}
		\hline
		$d$ & $n$  & Veronese & neighbourly & simplicial \\ \hline
		3 & 4  & 1        & 1          & 1          \\
		& 5  & 1        & 1          & 1          \\
		& 6  & 2        & 2          & 2          \\
		& 7  & 1        & 5          & 5          \\
		& 8  & 1        & 14         & 14         \\
		& 9  & 1        & 50         & 50         \\
		& 10 & 1        & 233        & 233        \\
		& 11 & 1        & 1249       & 1249       \\
		& 12 & 1        & 7595       & 7595       \\
		& & & &\\ \hline
	\end{tabular}
	\hspace{.5em}
	\begin{tabular}{|c|c|ccc|}
		\hline
		$d$ & $n$  & Veronese & neighbourly & simplicial \\ \hline
		4 & 5  & 1        & 1          & 1          \\
		& 6  & 2        & 1          & 2          \\
		& 7  & 5        & 1          & 5          \\
		& 8  & 6        & 3          & 37         \\
		& 9  & 5        & 23         & 1142      \\
		& 10  & 6        & 431         & 162004      \\
		& 11  & 6        & 13935         & ?      \\
		& 12  & 7        & \makecell[tc]{556061 \\or\\ 556062}         & ?      \\
		\hline
	\end{tabular}
	
	\vspace{1em}
	\begin{tabular}{|c|c|ccc|}
		\hline
		$d$ & $n$  & Veronese & neighbourly & simplicial \\ \hline
		5  & 6  & 1        & 1          & 1          \\
		& 7  & 2        & 1          & 2          \\
		& 8  & 8        & 2         & 8         \\
		& 9  & 9        & 125     & 322        \\
		& 10 & 10      & 159375 & ?      \\
		\hline
	\end{tabular}
	\hspace{.5em}
	\begin{tabular}{|c|c|ccc|}
		\hline
		$d$ & $n$  & Veronese & neighbourly & simplicial \\ \hline
		6  & 7  & 1        & 1          & 1          \\
		& 8  & 3        & 1         & 3         \\
		& 9  & 18        & 1     & 18        \\
		& 10 & 24      & 37 & ?      \\
		& 11& 27 &42099    &?\\ \hline
	\end{tabular}
	
	\centering
	\vspace{1em}
	\begin{tabular}{|c|c|ccc|}
		\hline
		$d$ & $n$  & Veronese & neighbourly & simplicial \\ \hline
		7 & 8  & 1        & 1         & 1         \\
		& 9  & 3        & 1     & 3        \\
		& 10 & 29      & 4 & 29      \\
		\hline
	\end{tabular}
	\end{minipage}
	\caption{Computational results for the number of combinatorial types of Veronese polytopes per dimension $d$ and number of vertices $n$. The number of general neighbourly and simplicial polytopes are displayed for comparison \cite[Table 1]{firsching17_realizabilityinscribabilitysimplicial}. See also \Cref{table:more-comb-types}.}
	\label{table:small-dimensions}
\end{table}

\addtocontents{toc}{\protect{\pdfbookmark[1]{\contentsname}{toc}}}
\renewcommand\contentsname{\vspace{2em} \\ Contents \\ \vspace{1em} }

{\hypersetup{hidelinks}
	\tableofcontents
}

\newpage

\section{Background}

\subsection{Factorization structures}

We recall results from the theory of factorization structures, which form the foundation for the theory of Veronese polytopes. A detailed and comprehensive account of this subject can be found in \cite{pucek2023factorization}. \\

To establish notation let $V_1,\ldots,V_d$, $d\geq2$, denote real/complex 2-dimensional vector spaces, and let $[d] = \{1,\dots,d\}$. For $j\in[d]$ and a $1$-dimensional subspace $\ell\subset V_j$ we define
\begin{align*}
	V=\bigotimes_{r=1}^d V_r
	\hspace{1cm}
	\text{and}
	\hspace{1cm}
	\Sigma_{j,\ell}=
	V_1\otimes\cdots\otimes V_{j-1}\otimes\ell\otimes V_{j+1}\otimes\cdots\otimes V_m,
\end{align*}
and denote the dual of $V$ by $V^*$ and the annihilator of $\Sigma_{j,\ell}$ in $V^*$ by $\Sigma_{j,\ell}^0$. \par
The projective space $\mathbb{P}(W)$ is viewed as the set of $1$-dimensional subspaces in the vector space $W$, equipped with the Zariski topology. Often, we identify $\ell\in\mathbb{P}(W)$ with the corresponding $1$-dimensional subspace of $W$, and for this reason, points of $\mathbb{P}(W)$ are also referred to as \textit{directions}. We say that a condition holds for a \textit{generic} point or \textit{generically} if there exists an open non-empty subset $U\subset\mathbb{P}(W)$ such that the condition holds at each point of $U$.

\begin{defn}\label{fs def}
Let $d$ be a positive integer. An injective linear map $\varphi:\mathfrak{h}\to V^*$ of a real/complex $(d+1)$-dimensional vector space $\mathfrak{h}$ into real/complex $V^*$ is called a \textit{factorization structure} of dimension $d$ if
\begin{align}\label{wfs def condition}
\dim \left( \varphi(\mathfrak{h}) \cap \Sigma_{j,\ell}^0 \right) = 1
\end{align}
holds for every $j\in[d]$ and generic $\ell\in\mathbb{P}(V_j)$.
An isomorphism of factorization structures is the commutative diagram
\begin{center}
\begin{tikzcd}
	\mathfrak{h}_1 \arrow[d, "\varphi_1"'] \arrow[rr, "\Phi"]                      &  & \mathfrak{h}_2 \arrow[d, "\varphi_2"] \\
	V_1^*\otimes\cdots\otimes V_m^* \arrow[rr, "(\phi_1\otimes\cdots\otimes\phi_m)\pi"] &  & W_1^*\otimes\cdots\otimes W_m^*      
\end{tikzcd},
\end{center}
where $\Phi$ and $\phi_j:V_{\sigma(j)}^*\to W_j^*$ are linear isomorphisms for all $j\in[d]$, and $\pi$ is a permutation of $[d]$ viewed as the braiding map
$V_1^*\otimes\cdots\otimes V_d^*\to V_{\pi(1)}^*\otimes\cdots\otimes V_{\pi(d)}^*$.
\end{defn}

\begin{rem}\label{the remark}
Setting $\pi=\text{id}$ and $\phi_j=\text{id}$, $j\in [d]$, shows that any two factorization structures with the same images are undistinguishable up to a choice of $\Phi$, which does not play a role in the defining condition \eqref{wfs def condition}. Thus, a factorization structure $\varphi$ can be identified with the subspace $\varphi(\mathfrak{h}) \subset V^*$.
\end{rem}

The defining equation \eqref{wfs def condition} of a factorization structure induces an assignment: a generic point $\ell\in\mathbb{P}(V_j)$ is mapped to
$\varphi^{-1} \left( \varphi(\mathfrak{h}) \cap \Sigma_{j,\ell}^0 \right) \in \mathbb{P}(\h)$. It can be shown that this map extends uniquely to a regular map $\psi_j: \mathbb{P}(V_j) \to \mathbb{P}(\h)$, i.e., a projective curve, and thus each factorization structure of dimension $d$ induces $d$ projective algebraic curves called \textit{factorization curves}.

\begin{example}[Veronese factorization structure]\label{SV ex}
A crucial example of a factorization structure is \textit{the Veronese factorization structure}, being
\begin{gather*}
\varphi: S^dW^* \to (W^*)^{\otimes d},
\end{gather*}
where $\dim W = 2$ and $\varphi$ is the canonical inclusion of symmetric tensors into tensors. The condition \eqref{wfs def condition} can be easily verified, since the dimension of
\begin{gather*}
\varphi(\h) \cap \Sigma_{j,\ell}^0 = S^dW^* \cap (W^*)^{\otimes j-1} \otimes \ell^0 \otimes (W^*)^{\otimes d-j} = (\ell^0)^{\otimes d}, \hspace{.4cm} j \in [d],
\end{gather*}
is one for any, and thus for generic, $\ell \in \mathbb{P}(W)$. This computation also yields factorization curves $\psi_j: \mathbb{P}(W) \to \mathbb{P}(S^dW^*)$,
\begin{gather*}\label{Veronese curves}
\varphi \circ \psi_j(\ell) = (\ell^0)^{\otimes d}, \hspace{1cm} j\in[d].
\end{gather*}
Observe that all factorization curves are defined globally and that they all coincide, i.e., $\psi_1 = \cdots = \psi_d =: \psi$. Moreover, $\psi$ is a regular map, i.e., given by homogeneous polynomials of the same degree in homogeneous coordinates on $\mathbb{P}(W)$ and $\mathbb{P}(S^dW^*)$. In fact, $\psi$ is the \emph{rational normal curve}. More concretely, as detailed in \Cref{sec:curve}, in particular coordinates on $\mathbb{P}(S^dW^*)$
and in an affine chart given by a non-zero $\xi \in S^d W$, the curve $\psi$ reads 
\begin{align*}
	t &\mapsto 
	\frac{(1,t,\ldots,t^d)}{\sum_{i=0}^{d} \xi_i t^i}.
\end{align*}
Note that for $\xi = (1,0,\ldots,0)$ we recover the moment curve $t \mapsto (1,t,\ldots,t^d)$, which is foundational to the theory of cyclic polytopes \cite{caratheodory1911variabilitatsbereich,gale1963neighborly}.

The following is a well-known fact about rational normal curves which we use extensively.

\begin{lemma}[{\cite{harris2013algebraic}}]\label{indep dir}
Any $k \leq d + 1$ directions on $\psi$ are linearly independent.
\end{lemma} 
\end{example}

\subsection{Segre-Veronese factorization structures}\label{sec:SV}

While this work primarily focuses on the geometry and combinatorics of the Veronese factorization structure, this subsection introduces a class of factorization structures that generalise the Veronese factorization structure. It aims to enhance our understanding of factorization structures. In the following we choose not to elaborate on the specific shapes of the defining tensors, since this is outside the scope of our study; interested readers can find extensive information on this topic in \cite{pucek2023factorization}. \\

Let $k$ be a positive integer. For a partition $d=d_1+\cdots+d_k$, $d_j\geq1$, and a fixed $j\in[k]$ we define the operator
$$\ins_j:
(W_j^*)^{\otimes d_j}\otimes\bigotimes_{\substack{i=1\\i\neq j}}^k (W_i^*)^{\otimes d_i}
\to
\bigotimes_{i=1}^k (W_i^*)^{\otimes d_i}$$
which acts on decomposable tensors by
\begin{align*}
\left(w_j^1\otimes\cdots\otimes w_j^{d_j}\right)
\otimes
\bigotimes_{\substack{i=1\\i\neq j}}^k
\left(w_i^1\otimes\cdots\otimes w_i^{d_i}\right)
\mapsto
\bigotimes_{i=1}^k
\left(w_i^1\otimes\cdots\otimes w_i^{d_i}\right),
\end{align*}
where $W_j$, $j\in [k]$, are vector spaces. Partitions $d=d_1+\cdots+d_p$ and $d=e_1+\cdots+e_q$ are considered to be the same if $\{d_1,\ldots,d_p\}=\{e_1,\ldots,e_q\}$, and distinct if they are not the same.

\begin{defn}\label{SV def}
For a partition $d = d_1+\dots+d_k$ of an integer $d\geq2$ and $2$-dimensional vector spaces $W_r$, $r\in [k]$, , let 
$\Gamma_j
\subset
\bigotimes_{r=1,r\neq j}^k(W_r^*)^{\otimes d_r}$,
$j\in[k]$, be $1$-dimensional subspaces such that
\begin{align}\label{SV image}
\sum_{j=1}^{k}
\ins_j
\left(
S^{d_j}W_j^*\otimes\Gamma_j
\right)
\end{align}
has dimension $d+1$, where $S^{d_j}W_j^*\subset(W_j^*)^{\otimes d_j}$ is viewed as the subspace of symmetric tensors. The \textit{standard Segre-Veronese factorization structure} is defined to be the canonical inclusion
\begin{align}\label{SV inclusion}
\sum_{j=1}^{k}
\ins_j
\left(
	S^{d_j}W_j^*\otimes \Gamma_j
\right)
\hookrightarrow
\bigotimes_{j=1}^k(W_j^*)^{\otimes d_j}.
\end{align}
Factorization structures corresponding to trivial partitions, $d=1+\cdots+1$ and $d=d$, are called \textit{Segre} and \textit{Veronese} respectively, the latter being $S^dW^* \hookrightarrow (W^*)^{\otimes d}$.  An element of the isomorphism class of the standard Segre-Veronese factorization structure is referred to as a Segre-Veronese factorization structure.
\end{defn}
\begin{rem}
To verify that \eqref{SV inclusion} defines a factorization structure we note that for $i\in[k]$ and generic $\ell\in\mathbb{P}(W_i)$ we have
\begin{align}\label{eq:correction1}
\varphi(\mathfrak{h})
\cap
\Sigma_{d_1+\cdots+d_{i-1}+q,\ell}^0
=
\ins_i
\left(
	(\ell^0)^{\otimes d_i}
	\otimes
	\Gamma_i
\right),
\end{align}
where $\varphi(\mathfrak{h})$ is \eqref{SV image}, $q\in[d_i]$ and $d_0$ is defined to be $0$. Note that there are at most finitely many $\ell\in\mathbb{P}(V_i)$ for which the dimension of the intersection in \eqref{eq:correction1} could be strictly larger than one, and, loosely speaking, this occurs when defining directions/tensors $\Gamma_j$, $j\neq i$, decompose at the $i$-th place. We found that
\begin{gather}\label{standard curves}
\varphi \circ \psi_{d_1+\cdots+d_{i-1}+q} (\ell) =
\ins_i
\left(
	(\ell^0)^{\otimes d_i}
	\otimes
	\Gamma_i
\right),
\hspace{.5cm} \ell\in\mathbb{P}(W_i),
\end{gather}
where $q\in[d_i]$ as above.
\end{rem}

\subsection{Compatible cones and polytopes}\label{sec:compatible-cones}

This subsection introduces and describes cones and polytopes compatible with a factorization structure.
For a thorough introduction to polyhedral geometry we refer to \cite{grunbaum1967convex,ziegler1993lectures}.

\begin{defn}\label{compatible polytope}
A full-dimensional and pointed convex polyhedral cone in $\mathfrak{h}$ is called \textit{compatible} with a factorization structure $\varphi : \mathfrak{h} \to V^*$ if its projectivised extremal rays lie on factorization curves of $\varphi$.
A polytope is called \textit{compatible} with a factorization structure $\varphi$ if it is a section of a cone compatible with $\varphi$ by an affine hyperplane.
\end{defn}

\noindent\textbf{Construction of compatible cones.} A way to construct a cone compatible with $\varphi$ is to fix points on factorization curves and 'de-projectivise' them.  
To describe this procedure rigorously, we start by fixing an arbitrary finite collection of general points $p_1,\ldots,p_n\in\mathbb{P}(\mathfrak{h})$. 
Dually, their annihilators $(p_j)^0 \subset \mathfrak{h}^*$, $j\in[n]$, give rise to the \textit{hyperplane arrangement} $\bigcup_{j=1}^n (p_j)^0$, splitting $\mathfrak{h}^*$ into a union of closed full-dimensional pointed convex polyhedral cones, which we call \textit{chambers}, which, in general, do not have $n$ facets. 
Observe that for $\sigma$ such a cone, bounded by $(p_{i_1})^0,\ldots,(p_{i_r})^0$, $r\leq n$, all lines $p_1,\ldots,p_n$ contain rays $p_1^+,\ldots,p_n^+$ which belong to the dual cone $\sigma^\vee$, since all functional $\alpha\in\sigma$ evaluate non-negatively on them, but the only extremal rays of $\sigma^\vee$ are $p_{i_1}^+,\ldots,p_{i_r}^+$. 
Recall that the dual of a full-dimensional pointed convex cone is itself a cone of the same type, and thus, if $p_{i_1},\ldots,p_{i_r}$ lie on factorization curves, then $\sigma^\vee$ is a cone compatible with the factorization structure. 
Note that if $\sigma$ is a convex union of neighbouring cones in $\mathfrak{h}^*$, it still is possible to obtain its description as above by omitting points whose annihilators pass through $\sigma$. 

\begin{example}\label{ex:arrangements}
\Cref{fig:arrangements} illustrates a choice of a chamber $\sigma$ in a hyperplane arrangement and its dual $\sigma^\vee$ in the dual line arrangement.
\Cref{fig:linearr} displays $1$-dimensional spaces $p_j = \text{span}\{e_{j-1}\}$, $j\in [3]$, and $p_4 = \text{span}\{(1,1,1)\}$, where $e_0,e_1,e_2$ is the standard basis of $\mathbb{R}^3$.
Dually, \Cref{fig:hyparr} shows the hyperplane arrangement consisting of their annihilators $(p_j)^0$, $j\in [4]$; $(p_4)^0$ depicted in purple, the others in light blue.
The choice of $\sigma$ in the hyperplane arrangement $\cup_{j=1}^4 (p_j)^0$ is indicated in green; facets of $\sigma$ are supported by hyperplanes whose normal directions are $p_1$, $p_2$ and $p_3$.
The dual cone $\sigma^\vee$ is presented again in green colour.
It consists of all functionals evaluating non-negatively on $\sigma$.
Correspondingly, its rays are generated by $e_0$, $e_1$ and $e_2$, while the ray generated by $(1,1,1)$ belongs into its interior.
	\begin{figure}
		\centering
		\begin{subfigure}[t]{0.49\textwidth}
			\centering
			\begin{tikzpicture}
			\node at (0,0) {\includegraphics[height=20em]{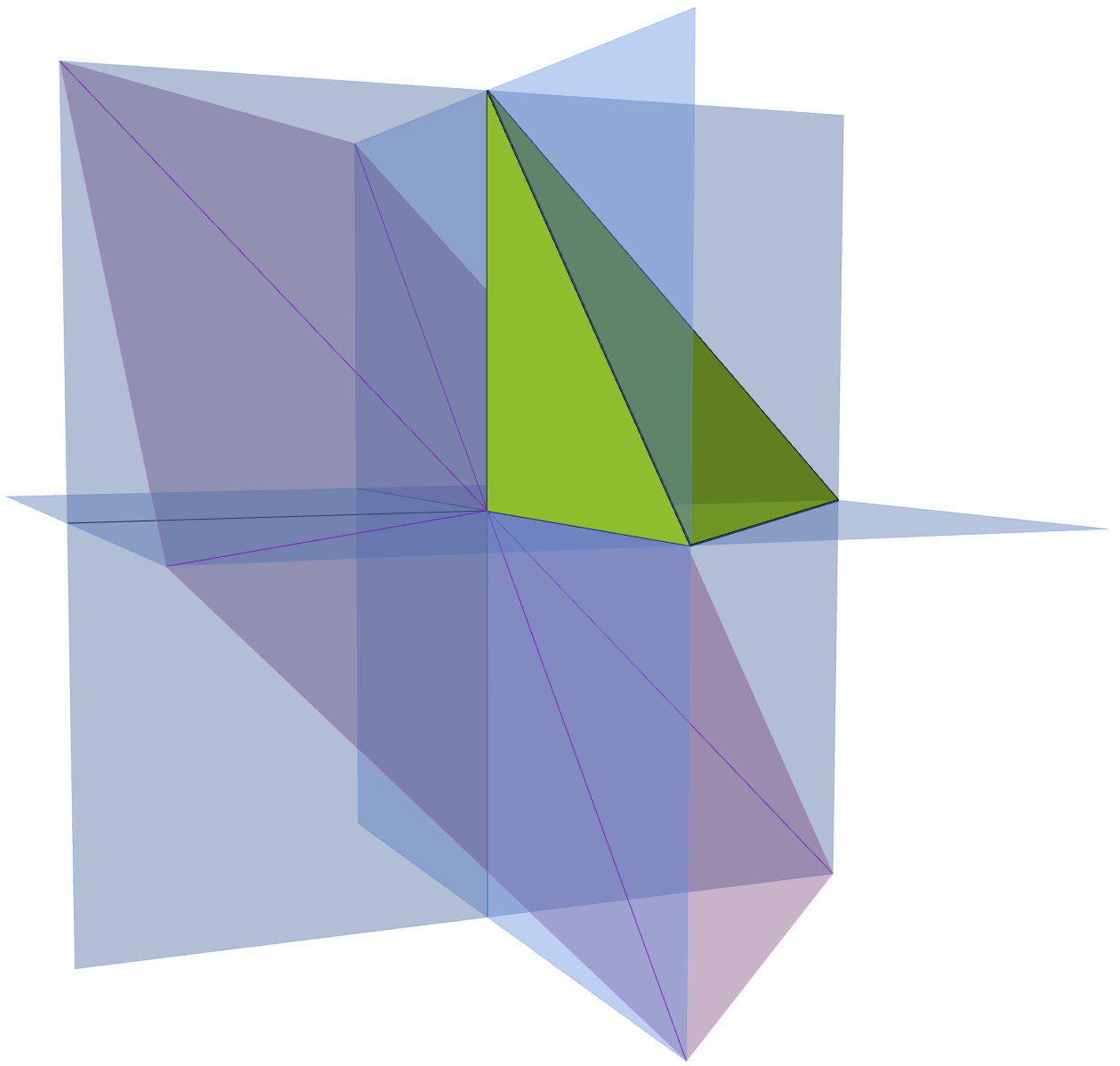}};
			\node at (0.7,2.2) {$\sigma$};
			\end{tikzpicture}
			\caption{Hyperplane arrangement}
			\label{fig:hyparr}
		\end{subfigure}
		\begin{subfigure}[t]{0.49\textwidth}
			\centering
			\raisebox{4.3em}{
				\begin{tikzpicture}
				\node at (0,0) {\includegraphics[height=15em]{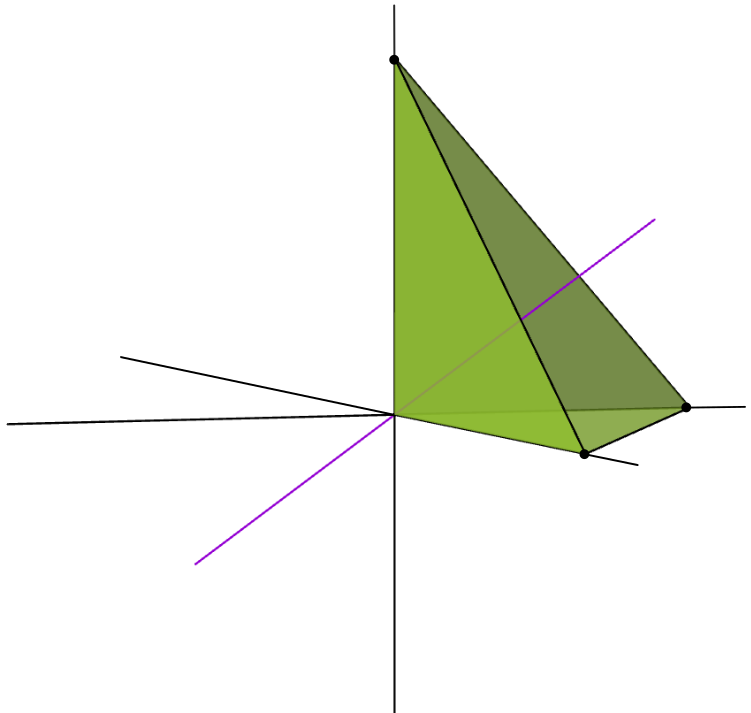}};
				\node at (2.65,-0.15) {$e_0$};
				\node at (1.75,-1.1) {$e_1$};
				\node at (-.1,2.45) {$e_2$};
				\node at (1.3,1.6) {$\sigma^\vee$};
				\end{tikzpicture}
			}
			\caption{Line arrangement}
			\label{fig:linearr}
		\end{subfigure}
		\caption{The dual arrangements from \Cref{ex:arrangements}.}
		\label{fig:arrangements}
	\end{figure}
\end{example}

\noindent\textbf{Construction of compatible polytopes.}
Projective points and their associated cones as described above are closely related to affine charts. Recall that an \textit{affine chart} on $\mathbb{P}(\h)$ is a choice of non-zero $\xi \in \h^*$ together with the induced map
\begin{align*}
U_\xi &\to \h \\
p &\mapsto \frac{p}{\langle \xi, p \rangle},
\end{align*}
where $U_\xi := \mathbb{P}(\h) \backslash \xi^0$, $\langle , \rangle$ denotes the standard contraction between elements of $\mathfrak{h}^*$ and $\mathfrak{h}$, and $p/\langle \xi, p \rangle$ is the vector given by the intersection of the $1$-dimensional space $p \subset \h$ and the affine hyperplane $\{ x \in \h \mid \langle \xi, x \rangle = 1\}$, i.e.,
\begin{gather*}
\left\langle \xi, \frac{p}{\langle \xi , p \rangle} \right\rangle = 1.
\end{gather*}
Observe that points $p_1,\ldots,p_n \in \mathbb{P}(\h)$ lie in $U_\xi$ if and only if $\xi \notin \bigcup_{j=1}^n (p_j)^0$. By fixing such $\xi$, we get, in the splitting of $\h^*$ by the hyperplane arrangement $\bigcup_{j=1}^n (p_j)^0$, a unique chamber $\sigma$ that $\xi$ belongs to. It follows that its dual is 
\begin{gather*}
\sigma^\vee = \cone \left(
	\frac{p_j}{\langle \xi , p_j \rangle} \in \h \hspace{.1cm} \bigg| \hspace{.1cm} j\in[n]
	\right).
\end{gather*}
As explained above, the extremal rays of $\sigma^\vee$ can be found among rays $\mathbb{R}_{\geq 0}\cdot p_j/\langle \xi, p_j \rangle$, $j\in[n]$. Additionally, because $\sigma^\vee$ is pointed, any vector $\xi$ from the interior $\interior(\sigma)$ determines a section $P_\xi$ of $\sigma^\vee$ by an affine hyperplane parallel to the annihilator $\xi^0$;
\begin{gather*}
P_\xi 
=
\sigma^\vee \cap \{ x \in \mathfrak{h} \hspace{.1cm} | \hspace{.1cm} \langle \xi, x \rangle = 1 \} 
=
\left\{ \frac{x}{\langle \xi, x \rangle} \in \mathfrak{h}  \ \bigg| \   x \in \sigma^\vee \right\}=
\conv \left( \frac{p_j}{\langle \xi , p_j \rangle} \ \bigg| \ j\in[n] \right).
\end{gather*}
Observe that for a fixed $\sigma$, the combinatorial type of polytopes $P_\xi$ parametrised by $\xi\in \interior(\sigma)$ is fixed (in fact, they are projectively equivalent). In sum, we found that points $p_1,\ldots,p_n \in \mathbb{P}(\mathfrak{h})$ and a choice of $\xi \in \mathfrak{h}^* \backslash \bigcup_{j=1}^n (p_j)^0$, determine the cone $\sigma$ uniquely by requiring $\xi \in \interior(\sigma)$, and therefore determine $\sigma^\vee$, its extremal rays, and $P_\xi \subset \sigma^\vee$. Assuming that the points lie on factorization curves we showed the following.
\begin{proposition}
Let $\varphi: \mathfrak{h} \to V^*$ be a factorization structure of dimension $d$ with distinct factorization curves $\psi_j:\mathbb{P}(W_j) \to \mathbb{P}(\mathfrak{h})$, $j\in[k]$, and let $\psi_j(\ell_{ji})$, $j\in[k]$, $i\in[n_j]$, $n_j \geq 0$, be a finite collection of points on the curves. If $\xi \in \mathfrak{h}^*$ does not lie in any of the annihilators $(\psi_j(\ell_{ji}))^0$, $j\in[d]$, $i\in[n_j]$, then
\begin{gather*}
P_\xi =
\conv \left( \frac{\psi_j(\ell_{ji})}{\langle \xi , \psi_j(\ell_{ji}) \rangle} \in \mathfrak{h} \hspace{.1cm} \bigg| \hspace{.1cm} j\in[k], i\in[n_j] \right)\
\end{gather*}
is a polytope compatible with the factorization structure $\varphi$. In particular, all compatible polytopes arise this way.
\end{proposition}

\begin{defn}
	A polytope compatible with the Veronese factorization structure is called a \emph{Veronese polytope}.
\end{defn}

\begin{rem}
We wish to remark that in \cite{pucek2023factorization} compatible polytopes are defined as duals of those from \Cref{compatible polytope}. Regardless of which cones and polytopes are called compatible with a factorization structure, and which are called their duals, all exhibit an explicit description and compatibility with the given factorization structure.
\end{rem}

\section{Geometry of Veronese polytopes} \label{geometry section}

This section focuses on polytopes compatible with the Veronese factorization structure. In a manner akin to cyclic polytopes, which represent a notable class of Veronese polytopes, their properties are given by a curve and feature a condition which generalizes Gale's evenness condition for cyclic polytopes. \\

\subsection{The curve explicitly}\label{sec:curve}

Recall from \Cref{SV ex} that the Veronese factorization structure $\varphi: \h \to V^*$ is given by $\h = S^dW^*$ and $V^* = (W^*)^{\otimes d}$, where $\dim W = 2$. All of its factorization curves agree, being the rational normal curve
\begin{gather*}
\psi(\ell)=(\ell^0)^{\otimes d}, \hspace{.5cm} \ell \in \mathbb{P}(W).
\end{gather*}

We wish to be more explicit and describe the curve $\psi$ in coordinates. We start by detailing choices of coordinates on the domain and target of $\psi$. To do so, we also briefly recall how a choice of a basis of a (2-dimensional) vector space induces bases on (symmetric) tensor powers of the space. Details can be found in introductory texts to multilinear algebra and tensors, for example in \cite{gorodentsev}.\\

From now on we consider a fixed basis $e_1,e_2$ of $W^*$, and let $e^1,e^2 \in W$ be the dual basis. These induce homogeneous coordinates on $\mathbb{P}(W^*)$ and $\mathbb{P}(W)$. For example, in these coordinates, a $1$-dimensional space $L \subset W^*$, which is the linear span of a vector $x e_1 + y e_2 \in W^*$, has the expression $[x:y]$, and its annihilator $L^0 \in \mathbb{P}(W)$, necessarily containing the vector $y e^1 - x e^2 \in W$, reads $[y:-x]$. We recall, since for any $r\in\mathbb{R}\backslash \{0\}$ the vector $r(xe_1 + ye_2)$ spans $L$ as well, the expression $[rx:ry]$ is also homogeneous coordinates for $L$, and thus $[x:y] = [rx:ry]$ describe the same projective point $L \in \mathbb{P}(W^*)$. This demonstrates that homogeneous coordinates are well-defined only up to multiplication by a non-zero scalar. \par
The affine chart on $\mathbb{P}(W^*)$ given by $e^1 \in W$ induces inhomogeneous coordinates; the image of $L$ in this affine chart is $(1,y/x)$ and the inhomogeneous coordinates $\mathbb{P}(W^*) \backslash (e^1)^0 \to \mathbb{R}$ are defined by $L \mapsto y/x$. The only point not covered by these coordinates is the span of $e_2$. Similarly, and by convention, $-e_2 \in W^*$ induces inhomogeneous coordinates on $\mathbb{P}(W) \backslash (e_2)^0$ in which we have $L^0 \mapsto y/x$.
 \par

Additionally, we have induced bases of the  $(d+1)$-dimensional vector spaces $\h=S^dW^*$ and $\h^*=S^dW$, being the standard basis for symmetric tensors on $W^*$ and $W$. To describe these more explicitly, recall that a permutation $\pi \in S_d$ induces the linear map $\pi: (W^*)^{\otimes d} \to (W^*)^{\otimes d}$ defined as the linear extension of the assignment 
$v_1 \otimes \cdots \otimes v_d \mapsto v_{\pi(1)} \otimes \cdots \otimes v_{\pi(d)}$ on decomposable tensors. In words, it permutes elements in the tensor product.
The basis of $S^dW^*$ then reads
\begin{gather}\label{sym tensors basis}
\epsilon_j = \frac{1}{(d-j)!} \sum_{\pi \in S_d} 
\pi 
\left( (e_2)^{\otimes j} \otimes (e_1)^{\otimes (d-j)} \right),
\hspace{1cm}
j=0,\ldots,d,
\end{gather}
and similarly for $S^dW$, which we denote by $\epsilon^j$, $j=0,\ldots,d$, since it is dual to \eqref{sym tensors basis}. 
Finally, the induced basis for $V^*=(W^*)^{\otimes d}$ consists of each summand in \eqref{sym tensors basis}, and similarly for $V=W^{\otimes d}$.\\

Now we describe $\psi$ in a chart. Let $\xi = \sum_{j=0}^d \xi_j \epsilon^j \in \h^* = S^dW$ be an affine chart and let the annihilator $\ell^0 \in \mathbb{P}(W^*)$ of $\ell \in \mathbb{P}(W)$ have homogeneous coordinates $[x:y]$. We wish to describe the assignment
\begin{align*}
\mathbb{P}(W) &\to \h \\
\ell &\mapsto \frac{\psi(\ell)}{\langle \psi(\ell), \xi \rangle},
\end{align*}
i.e., the curve $\psi$ in the affine chart $\xi$, in the coordinates described above. We find
\begin{align*}
\frac{\psi(\ell)}{\langle \psi(\ell), \xi \rangle} =
\frac{(\ell^0)^{\otimes d}}{\langle (\ell^0)^{\otimes d}, \xi \rangle} =
\frac{(x e_1 + y e_2)^{\otimes d}}{\langle (x e_1 + y e_2)^{\otimes d}, \xi \rangle},
\end{align*}
and
\begin{align*}
(x e_1 + y e_2)^{\otimes d} = \sum_{j=0}^d x^{d-j}y^j \epsilon_j \in S^dW^*.
\end{align*}
Therefore, the curve $\psi / \langle \psi, \xi \rangle$, in homogeneous coordinates on $\mathbb{P}(W)$ and in coordinates on  $S^dW^*$ given by the basis, reads
\begin{align*}
[y:-x] \mapsto 
\frac{ (x^d, x^{d-1}y, \ldots, x y^{d-1}, y^d) }{\sum_{j=0}^d \xi_j x^{d-j} y^j},
\end{align*}
which, in the inhomogeneous coordinates on $\mathbb{P}(W) \backslash (e_2)^0$ as above, gives the curve
\begin{align}
	\R &\to \R^{d+1} \nonumber \\ \label{inhom curve}
t &\mapsto \frac{(1,t,\ldots,t^d)}{\sum_{j=0}^d \xi_j t^j}.
\end{align}
The only point of $\mathbb{P}(W)$ not captured by inhomogeneous coordinates is $(e_2)^0 = e^1 = [1: 0]$, which, under $\psi$, corresponds to $[0:\cdots : 0: 1]$.
If the latter point does not belong to the affine chart $\xi$, i.e., if $\xi_d = 0$, then \eqref{inhom curve} describes the entire curve $\psi / \langle \psi, \xi \rangle$.
Otherwise, the representation \eqref{inhom curve} covers all points of $\psi / \langle \psi, \xi \rangle$ except one, $(0,\ldots,0,1/\xi_d)$.
Observe that if $\xi = \epsilon_0$, then \eqref{inhom curve} recovers the moment curve.

\begin{example}\label{ex:ruled-surface}
\Cref{fig:ruled-surface} illustrates the curve $\psi$ in two affine charts $\xi = (1,0,0)$ and $\xi' = (-1,0,1)$.
The ruled surface (in purple) represents $\psi$ as a set of lines in $\mathbb{R}^3$ rather than a set of projective points in $\mathbb{RP}^2$, and is given by the equation $x_0 x_2 - x_1^2 = 0$.
Intersecting it with the affine hyperplane $\langle \xi, (x_0,x_1,x_2) \rangle = 1$, i.e., considering the image of $\psi$ in the affine chart $\xi$, results in the moment curve $(1,t,t^2)$ (in blue).
The affine chart $\xi'$ provides us with the curve $(1,t,t^2)/(t^2 - 1)$ (in red), which, in contrast to the moment curve, is not connected.
The orange lines represent general lines on the ruled surface $\psi$, having a unique intersection point with each of the two curves.
	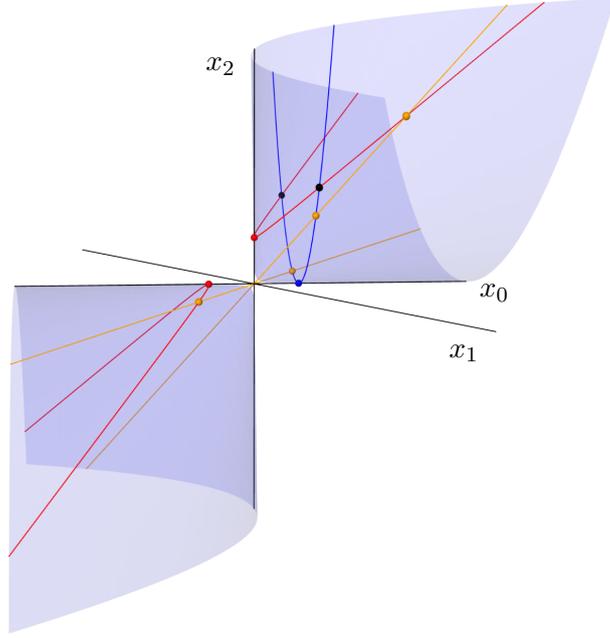
\begin{figure}
		\begin{tikzpicture}
			\node at (0,0) 
			{\includegraphics[height=22em]{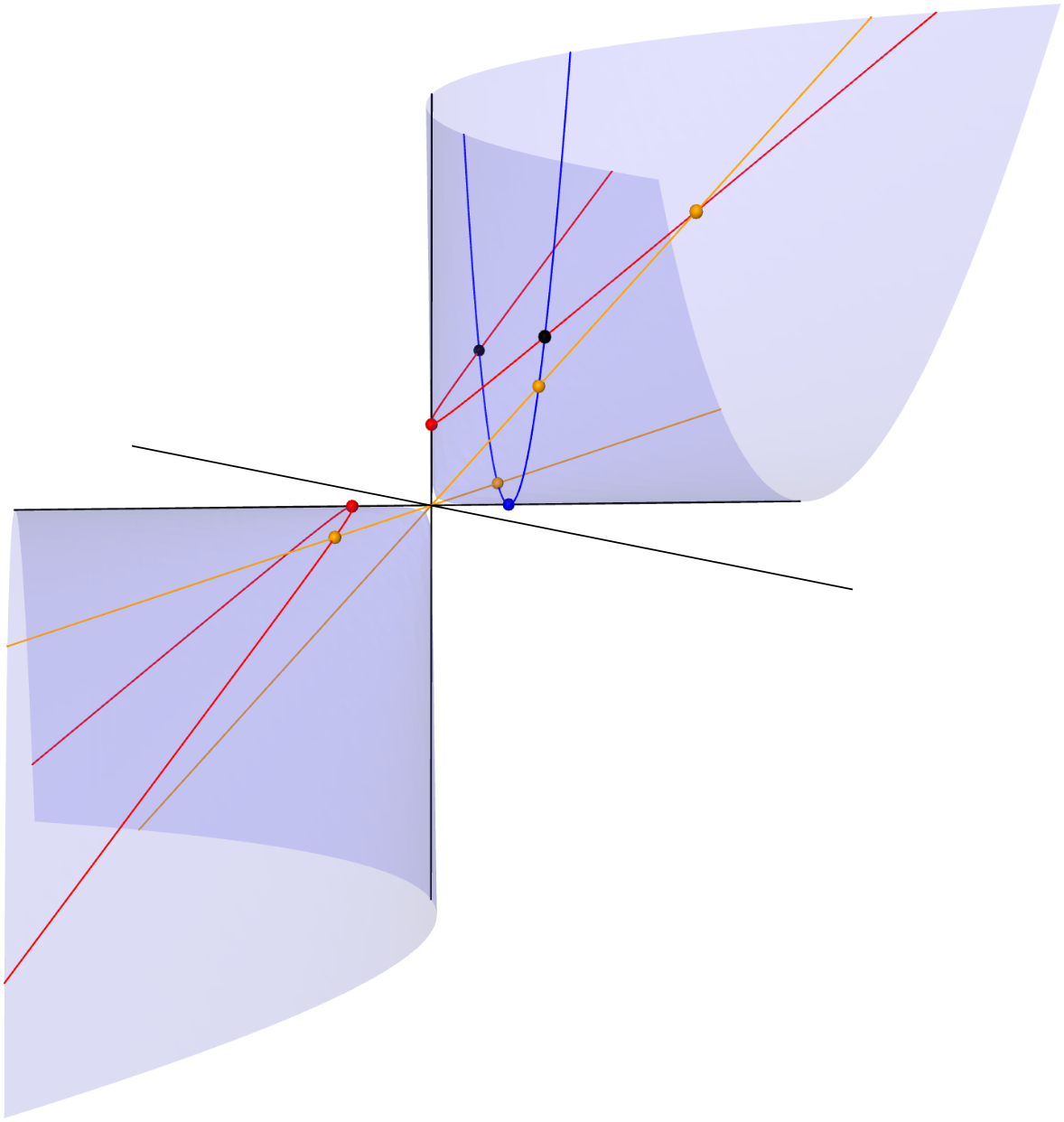}};
			\node at (2.4,0.3) {$x_0$};
			\node at (2,-0.5) {$x_1$};
			\node at (-1.2,3.3) {$x_2$};
		\end{tikzpicture}
		\caption{The curves and ruled surface from \Cref{ex:ruled-surface}}
		\label{fig:ruled-surface}
	\end{figure}
\end{example}

\begin{rem}\label{intrinsic rem}
It is exceptional that a general factorization structure admits bases of $\h$ or $\h^*$ in which computations are be transparent or illuminating. Yet, the curve $\psi$ can be expressed in terms intrinsic to factorization structure theory as follows. Recall that $\varphi^t$ is surjective, and thus for any $\xi \in \h^*$ there exists $\hat{\xi} \in V$ so that $\varphi^t \hat{\xi} = \xi$, and any other $\hat{\xi}'$ such that $\varphi^t \hat{\xi}' = \xi$ is of the form $\hat{\xi}' = \hat{\xi} + X$ for some $X \in (\varphi (\h))^0 \subset V$. By embedding $\psi$ via $\varphi$ we find
\begin{align*}
\frac{\varphi \circ \psi(\ell)}{\langle \psi(\ell), \xi \rangle} =
\frac{\varphi \circ \psi(\ell)}{\langle \psi(\ell), \varphi^t\hat{\xi} \rangle} =
\frac{\varphi \circ \psi(\ell)}{\langle \varphi \circ \psi(\ell), \hat{\xi} \rangle} =
\frac{(x,y)^{\otimes d}}{\langle (x,y)^{\otimes d}, \hat{\xi} \rangle},
\end{align*}
and observe that it does not depend on the choice of $\hat{\xi}$ since
\begin{align*}
\frac{\varphi \circ \psi(\ell)}{\langle \varphi \circ \psi(\ell), \hat{\xi} + X \rangle} =
\frac{\varphi \circ \psi(\ell)}{\langle \varphi \circ \psi(\ell), \hat{\xi} \rangle}.
\end{align*}
In inhomogeneous coordinates as above, $\varphi \circ \psi$ in the chart reads
\begin{gather*}
t \mapsto \frac{(1,t)^{\otimes d}}{\langle (1,t)^{\otimes d}, \hat{\xi} \rangle} =
\frac{(1,t)^{\otimes d}}{\sum_{i=0}^d \xi_i t^i}.
\end{gather*}
All properties of polytopes compatible with the Veronese factorization structure presented in this section are proven using this intrinsic description.
\end{rem}

Since $\psi$ is a degree $d$ curve, it follows that any $d$-dimensional Veronese polytope is \emph{simplicial}, i.e., every facet is a $(d-1)$-dimensional simplex. This has already been observed in \cite{pucek2023factorization}, here we provide a proof for completeness.

\begin{proposition}\label{prop:simplicial}
	Every facet of a $(d+1)$-dimensional cone compatible with the Veronese factorization structure of dimension $d$ contains exactly $d$ extremal rays. In particular, a $d$-polytope compatible with the Veronese factorization structure is simplicial.
\end{proposition}
\begin{proof}
	Any facet-supporting hyperplane $H$ of a $(d+1)$-dimensional cone is spanned by $d$ linearly independent extremal rays, and because we work with a compatible cone, these lie on $1$-dimensional spaces $\psi(\ell_j)$, $j\in[d]$, for some $\ell_j \in \mathbb{P}(W)$. We need to show that $H$ does not contain any other extremal rays. 
	This follows from the fact that the projective hyperplane $\mathbb{P}(H) \subset \mathbb{P}(\mathfrak{h})$ intersects the degree $d$ curve $\psi$ in at most $d$ points.
	Alternatively, this follows from \Cref{indep dir}, since any $d+1$ points are linearly independent, and thus do not lie in a common hyperplane.
\end{proof}

It is known that the convex hull of $n \geq d+1$ points on a $d$-order curve is combinatorially equivalent to the $d$-dimensional cyclic polytope $C_d(n)$ on $n$ vertices \cite{cordovil00_cyclicpolytopesoriented}.
This result is used and cited repeatedly in the literature,
however, the definition of a $d$-order curve frequently omits
continuity and injectivity, or at least non-constancy, properties on which proofs in these publications rely.

\begin{defn}
Let $I \subset \mathbb{R}$ be a non-trivial interval and $V$ an affine space. A continuous injective map $\alpha: I \to V$ is a \textit{$d$-order curve} if every affine hyperplane in $V$ intersects $\im\alpha$ in at most $d$ points.
\end{defn}


We now characterize $\xi \in \mathfrak{h}^*$ for which $\psi / \langle \psi, \xi \rangle$ is a $d$-order curve, implying that associated Veronese polytopes are combinatorially equivalent to cyclic polytopes.

\begin{proposition} \label{d order curve}
Let $\psi: \mathbb{P}(W) \to \mathbb{P}(S^dW^*)$ be the factorization curve of the Veronese factorization structure.
The curve $\psi/ \langle \psi, \xi \rangle$ is a connected curve, and hence a $d$-order curve, if and only if $\xi \in \mathfrak{h}^*$ lies on the curve
\begin{gather}\label{xi curve}
[s:t] \mapsto \left[ (-1)^d {d \choose 0} s^d: (-1)^{d-1} {d \choose 1} s^dt: \cdots: -{d \choose d-1} st^{d-1}: {d \choose 0} t^d \right].
\end{gather}
\end{proposition}
\begin{proof}
Since the image of $\psi:\mathbb{P}(W) \to \mathbb{P}(\mathfrak{h})$ is a homeomorphic to a circle, the curve $\psi/ \langle \psi, \xi \rangle$ is connected if and only if $\im \psi \cap \xi^0$ is connected.
Because $| \im \psi \cap \xi^0 | \leq \deg \psi = d$, $\im \psi \cap \xi^0$ is connected if and only if the cardinality is zero or one.
Using the coordinates as above, the curve $\psi$ reads
\begin{gather*}
\psi([y:-x]) = [x^d: x^{d-1} y: \cdots : x y^{d-1}: y^d].
\end{gather*}
We look for all non-zero $\xi \in \mathfrak{h}^*$ such that
\begin{gather}\label{yay}
\sum_{j=0}^d
\xi_j x^{d-j} y^j = 0
\end{gather}
has zero or one solution, i.e., the solution set in $\mathbb{P}(W)$ is at most one point. 
To perform the computation we use two charts on $\mathbb{P}(W)$: one is
$\{ [y:-x] \mid x\neq 0 \} \to \mathbb{R}$, $[y:-x] \mapsto y/x$, given by $-e_2 \in W^*$, and the other is $\{ [y:-x] \mid y\neq 0 \} \to \mathbb{R}$, $[y:-x] \mapsto x/y$, given by $-e_1 \in W^*$. \par
If $x\neq0$, \eqref{yay} reads
\begin{gather}\label{yay1}
\sum_{j=0}^d
\xi_j (y/x)^j = 0,
\end{gather}
which has zero solutions if and only if $\xi_1 = \dots =\xi_d = 0$, and one solution if and only if there exists $c,t_0 \in \mathbb{R}$, $c \neq 0$, such that \eqref{yay1} can be written as $c(y/x - t_0)^d = 0$, i.e.,
\begin{gather}\label{xi on the curve 1}
\xi_{j} = c {d \choose j} (-t_0)^{d-j}, \hspace{.2cm} j = 0, \ldots, d.
\end{gather}
In the latter case the solution is $x=u$ and $y=ut_0$, where $u \in \mathbb{R} \setminus \{0\}$. \par
When $y\neq0$, the equation \eqref{yay} reads
\begin{gather*}
\sum_{j=0}^d
\xi_j (x/y)^j = 0,
\end{gather*}
which has zero solutions if and only if  $\xi_0 = \dots =\xi_{d-1} = 0$, and one solution if and only if
\begin{gather}\label{xi on the curve 2}
\xi_{d-j} = c' {d \choose d-j} (-s_0)^{d-j}, \hspace{.2cm} j = 0, \ldots, d
\end{gather}
for some $s_0, c' \in \mathbb{R}$, $c'\neq0$.
The solution is $x=us_0$ and $y=u$, $u \in \mathbb{R} \setminus \{0\}$. \par
We observe that if $\xi$ does not provide a solution in the chart $x\neq0$ or $y\neq0$, then it gives a unique solution $[x:y] = [0:1]$ or $[x:y] = [1:0]$ in the chart $y\neq0$ or $x\neq0$, respectively.
At the intersection of these charts, i.e., when $xy\neq0$, the unique solution $[x:y] = [1:t_0]$ determined by $\eqref{xi on the curve 1}$ corresponds to the unique solution $[x:y] = [s_0:1]$ of \eqref{xi on the curve 2} for $s_0=1/t_0$. Therefore, a solution always exists, and one can observe that this data define the projective curve in $\mathbb{P}(\mathfrak{h}^*)$ given by \eqref{xi curve}, being the rational normal curve, whose points parametrise $\xi$ such that $|\psi \cap \xi^0| = 1$.
Any such $\xi$ transforms \eqref{yay} into
\begin{gather}
(yt - xs)^d = 0,
\end{gather}
and the unique solution is $[x:y] = [t:s]$.
\end{proof}

\subsection{Generalised Gale condition}

We fix a general $d$-dimensional Veronese polytope $P_\xi$. Hence
\begin{gather}\label{polytope Veronese}
P_\xi =
\conv \left( \frac{\psi(\ell_j)}{\langle \xi, \psi(\ell_j) \rangle} \in S^dW^* \hspace{.1cm} \bigg| \hspace{.1cm} j\in [n] \right),
\end{gather}
where $\xi \in \h^* = S^dW$ is an affine chart, and $n\geq d+1$. Recall that it is not guaranteed that a \textit{generating point} $\psi(\ell_j) / \langle \xi, \psi(\ell_j) \rangle$ is a vertex for each $j\in [n]$. Despite this, we derive a generalisation of the Gale evenness condition which detects facets of $P_\xi$, and subsequently identifies vertices. \\

We wish to expresses $P_\xi$ using inhomogeneous coordinates as in \eqref{inhom curve}, but it may happen that one generating point cannot be parametrised this way, namely if it lies on the 1-dimensional space generated by $(0,\ldots,0,1)$ or, equivalently, if there exist $j\in[n]$ such that $\ell_j^0$ has homogeneous coordinates $[1:0]$.
This can be avoided by choosing the second basis vector $e_2$ of $W^*$ so that it does not lie on any of directions $\ell_j^0 \in \mathbb{P}(W^*)$, $j \in [n]$.
Indeed, such a choice is always possible since we consider only finitely many points, and it ensures that for $\ell_j^0$ with homogeneous coordinates $[x_j:y_j]$ we have $x_j \neq 0$, $j \in [n]$.
Note that we could also choose a basis of $W^*$ so that $x_j \neq 0$ and $y_j \neq 0$. Therefore, $P_\xi$ in these coordinates is
\begin{gather}\label{Veronese polytope}
P_\xi(T) =
\conv \left(
	\frac{\moment(t_j)}{\sum_{i=0}^{d} \xi_i t_j^i} \in \mathbb{R}^{d+1}
	\hspace{.1cm} \bigg| \hspace{.1cm}
	j\in [n]
\right),
\end{gather}
where
$\moment(t) = (1,t,\dots,t^d)$ and
 $T=\{ t_j \in \mathbb{R} \mid j\in[n] \}$ denotes the set of inhomogeneous coordinates of $\ell_j^0$, $j\in [n]$. Observe that for $\xi=\epsilon_0$, i.e., $\xi_i = \delta^i_0$, 
we recover definition of cyclic polytopes, where $\delta^i_j$ denotes the Kronecker delta. 
On occasions, we do not distinguish between $P_\xi$ from \eqref{polytope Veronese} and $P_\xi(T)$ from \eqref{Veronese polytope}, although, strictly speaking, they are subsets of distinct spaces. This is justified by the fact that $P_\xi(T)$ is merely a coordinate representation of $P_\xi$ via a linear isomorphism, and thus $P_\xi$ and $P_\xi(T)$ share the same properties. \\

\begin{rem}
We wish to remark that every Veronese polytope $\poly(T)$ (as in \eqref{Veronese polytope}) is defined with respect to a possibly different set of inhomogeneous coordinates.
For each polytope $P_\xi$ (as in \eqref{polytope Veronese}), these coordinates are chosen so that they are defined on all generating points of $P_\xi$.
This implies subtleties when considering families of Veronese polytopes.
More precisely, let $\ell(P_\xi) := \{ \ell_j \in \mathbb{P}(W) \mid j \in [n] \}$ be the set of directions defining the generating points of $P_\xi$ from \eqref{polytope Veronese}, and let $I$ be a (possibly infinite) family of Veronese polytopes in the sense of \eqref{polytope Veronese}.
If the set $\mathbb{P}(W) \backslash \bigcup_{P \in I} \ell(P)$ is non-empty, then common inhomogeneous coordinates for all polytopes $P \in I$ can be chosen by fixing any of its elements as the second basis vector $e_2$ of $W^*$.
\end{rem}

The following remark complements the affine point of view from above by a projective one.

\begin{rem}\label{proj pic Gale}
Observe that the Veronese polytope $P_\xi$ induces the following two sets of points on the curve $\psi$:
$\{\psi(\ell_j) \mid j\in[n]\}$,
and $\{\psi(\ell) \mid \psi(\ell) \subset \xi^0 \ \exists \ell \in \mathbb{P}(W) \}$. The latter has cardinality $r$ at most $d$, since it consists of intersection points of the degree $d$ curve $\psi$ with the hyperplane $\xi^0$. Using the bijectivity of $\psi$ we transfer these sets onto $\mathbb{P}(W)$, being (topologically, smoothly, algebraically) isomorphic with the circle $\mathbb{S}^1$, thus obtaining $\mathbb{S}^1$ with $n+r$ marked points. In particular, the former set of $n$ points corresponds to $\ell_j \in \mathbb{P}(W)$, $j\in [n]$. Since these $r$ points in the affine chart $\xi$ correspond to points of $\psi$ at infinity, one can interpret this as $\mathbb{S}^1$ being divided into $r$ arcs, on which we have altogether $n$ points. When we use inhomogeneous coordinates on $\mathbb{P}(W)$, there is another special point on $\mathbb{S}^1$, namely the unique point which is not parametrised by inhomogeneous coordinates -- usually called a point at infinity as well. The choice of basis for $W$ implies that this point is distinct from $\ell_j$, $j\in[n]$, however it can be at infinity with respect to $\xi$.
\end{rem}

Now we derive a generalised Gale evenness condition, characterising which generating points form a facet. Since by \Cref{indep dir} any $d$ distinct points on $\psi$ are linearly independent, any $d$ generating points determine a unique hyperplane. Geometrically, to decide if any such is a facet-supporting hyperplane, it is enough to understand if all generating points of $P_\xi$ lie on one of its sides, thus implying that the whole polytope lies on the same side. Analytically, we need to fix a normal vector of the hyperplane, compute its contraction with all generating points, and decide if the contractions have the same sign. \par
As detailed in \cite{pucek2023factorization}, the shape of the normal direction follows naturally from the factorization structure theory. Here, we only verify that $\varphi^t \otimes_{j\in J} \ell_j$ is the normal direction to the hyperplane $H$ given by $\psi_j(\ell_j)$, $j\in J$, where $J = \{j_1,\dots,j_d\} \subset \{1,\ldots,n\}$ is of cardinality $d$. 
We start by choosing a vector on $\varphi^t \otimes_{j\in J} \ell_j$ to be
\begin{gather*}
\frac{\varphi^t \bigotimes_{j\in J} \ell_j}{\langle \varphi^t \bigotimes_{j\in J} \ell_j , g \rangle},
\end{gather*}
where $g \in \h$ is the auxiliary affine chart on $\mathbb{P}(\h^*)$ determined by $\varphi g = (e_2)^{\otimes d}$. We use inhomogeneous coordinates as in \Cref{intrinsic rem} to find the contraction of the normal vector with a general point on the curve
\begin{align}\label{Gale 1 curve}
\lambda_{\xi,S}(t) &:=
\left \langle \frac{\varphi^t \bigotimes_{j\in J} \ell_j}{\langle \varphi^t \bigotimes_{j\in J} \ell_j , g \rangle}, \hspace{.1cm}
\frac{\psi(\ell)}{\langle \psi(\ell), \xi \rangle} \right \rangle = \nonumber
\left \langle \bigotimes_{j\in J} (-t_j,1), \hspace{.1cm}
\frac{(1,t)^{\otimes d}}{\sum_{i=0}^{d} \xi_i t^i} \right \rangle \\ 
&=\frac{\prod_{j \in J} (t - t_j)}{\sum_{i=0}^{d} \xi_i t^i} =
\frac{1}{\sum_{i=0}^{d} \xi_i t^i}
\sum_{i=0}^d (-1)^i t^{d-i} \sigma_i(t_{j_1},\ldots,t_{j_d})
\end{align}
where $\sigma_i$ is the $i$-th elementary symmetric polynomial, $\sigma_0:=1$, and 
$S = \{t_{j} \mid j \in J\}$. The expression \eqref{Gale 1 curve} is a rational function in $t$ vanishing at point of $S$, thereby showing that $\varphi^t \otimes_{j\in J} \ell_j$ is indeed the normal direction to $H$. Note that the normal vector can be expanded in the basis of $\h^* = S^dW$ either directly or it can be read off of \eqref{Gale 1 curve}; 
its $i$-th component is $(-1)^{d-i}\sigma_{d-i}(t_{j_1},\ldots,t_{j_d})$, $i=0,\ldots,d$. 
Moreover, the signs of $\lambda_{\xi,S}(t)$, 
$t \in T \backslash S$, determine if the hyperplane given by $S$ is facet-supporting. Since Veronese polytopes are simplicial, we showed the generalised Gale evenness condition.

\begin{thm}[Gale condition for Veronese polytopes]\label{geometric Gale}
	Let $P_\xi(T)$ be a Veronese polytope as in \eqref{Veronese polytope}, and let $S \subset T$ be of cardinality $d$. The unique hyperplane determined by the points
	\[
		\frac{\moment(t)}{\sum_{i=0}^d \xi_i t^i} \ , \quad t \in S,
	\]
	 is a facet-supporting hyperplane of $P_\xi(T)$ if and only if the values of $\lambda_{\xi,S}(t)$ at points $t \in T \backslash S$ have equal signs.
	 Moreover, this hyperplane does not contain any other vertices.
\end{thm}

As observed before, $\poly(T)$ is the cyclic polytope $C_d(n)$ for $\xi=\epsilon_0$. In this case \Cref{geometric Gale} turns into Gale's evenness condition for cyclic polytopes as follows.

\begin{corollary}[{\cite{gale1963neighborly}}]
	Let $\xi = \epsilon_0$. The unique hyperplane determined by $\moment(t_j), t_j \in S$ is a facet supporting-hyperplane of the cyclic polytope $C_d(n) = P_{\epsilon_0}(T)$ if and only if the polynomial $\lambda_{\epsilon_0,S}$ has constant sign on $T\setminus S$, or, equivalently, if and only if any two elements of $T \setminus S$ are separated by an even number of elements from $S$ in the sequence $\{t_1,\dots,t_n\}$.
\end{corollary}

\begin{rem}
Observe that $\lambda_{\xi,S}$ from \eqref{Gale 1 curve} extends to a real-valued map on $\mathbb{P}(W) \backslash \{ \ell \in \mathbb{P}(W) \mid \psi(\ell) \subset \xi^0 \}$, whose zeros are exactly at $\ell_j$, $j\in J$. In fact, the domain of this map is the circle without at most $d$ points (cf. \Cref{proj pic Gale}). 
\end{rem}

In the light of \Cref{geometric Gale}, we make the following definition.

\begin{defn}
Let $T = \{ t_1,\ldots,t_n \} \subset \mathbb{R}$ be of cardinality $n$, and let $P_\xi(T)$ be a Veronese $d$-polytope.
We say that $S \subset T$ of cardinality $d$ \textit{corresponds to a facet} of $P_\xi(T)$ if $\nu_d(t) / \langle \xi, \nu_d(t) \rangle$, $t \in S$, are vertices of a facet of $P_\xi(T)$.
\end{defn} 

\subsection{Signed and $\sigma$-Decompositions}
\label{sec:decompositions}

In this section we make several observations regarding the Gale condition for Veronese polytopes from \Cref{geometric Gale}. These observations will allow us to derive a combinatorial description of the facial structure of Veronese polytopes in \Cref{sec:combinatorics}.

Let $\bigcup_{j=1}^n ( \psi(\ell_j) )^0 = \bigcup_{j=1}^n ( \moment(t_j) )^0 $ be the hyperplane arrangement as described in \Cref{sec:compatible-cones}, and $\sigma$ its chamber associated with $\poly(T)$, in particular $\xi \in \interior(\sigma)$. As we observed previously, the combinatorial type of $\poly(T)$ is entirely determined by the chamber $\sigma$ and does not depend on its individual elements such as $\xi$; the natural map sending a point $v$ to its ray $\R_{\geq 0} \cdot  v$ provides an isomorphism of face lattices of $\poly(T)$ and $\sigma^\vee$.
This implies that for $S$ corresponding to a facet of $\poly(T)$, the values
$\sgn \circ \lambda_{\xi,S}(t_j)$, $j\in [n]$, remain the same for any $\xi \in \interior(\sigma)$.
Denoting polynomials
\begin{gather}\label{eq:polynomials}
p_S(t) := \prod_{t_j \in S} (t - t_j)
\hspace{.2cm}
\text{ and }
\hspace{.2cm}
q_\xi(t) := \sum_{i=0}^d \xi_i t^i = \langle \xi, \nu_d(t) \rangle,
\end{gather}
the rational function $t \mapsto \lambda_{\xi,S}(t)$ reads $\lambda_{\xi,S}(t) = p_S(t)/q_\xi(t)$, and 
$$\sgn \circ \lambda_{\xi,S} = (\sgn \circ p_S) \cdot (\sgn \circ q_\xi).$$
For $S \subset T$ fixed, since $\sgn \circ \lambda_{\xi,S}(t_j)$ depends on $\sigma$ only, and $\sgn \circ p_S$ does not depend on $\xi$ at all,
it must be that $\sgn \circ q_\xi (t_j)$, $j\in [n]$, depends on $\sigma$ only, i.e., is independent on the exact choice $\xi \in \interior(\sigma)$.
We can thus define
\begin{gather*}\label{eq:sign-cone}
\sgn_\sigma(t_j) := \sgn \circ q_\xi (t_j).
\end{gather*}
Because $q_\xi$ is a polynomial of degree at most $d$, we have an induced decomposition of $T$ into a union of $k+1$ discrete intervals $I_1,\dots,I_k$, $k \leq d$, such that
\begin{gather*}
\forall s,t \in I_i:
\hspace{.2cm}
\sgn_\sigma(s) = \sgn_\sigma(t),
\end{gather*}
and for $i=1,\ldots,k$
\begin{gather*}
\forall t \in I_i
\hspace{.2cm}
\forall s \in I_{i+1}:
\hspace{.1cm}
\sgn_\sigma(s) = - \sgn_\sigma(t)
\end{gather*}
holds.
In words, intervals $I_i$ carry a sign determined by $\sigma$ which alters with $i$.
	We wish to relate to the theory of oriented matroids by noting that the vector $(\sgn_{\sigma}(t_1),\dots,\sgn_\sigma(t_n))$ is the \emph{signed covector} corresponding to $\sigma$ in the oriented matroid associated to the hyperplane arrangement $\bigcup_{j=1}^n ( \moment(t_j) )^0$ .

\begin{defn}\label{def:sigma-decomposition}
A \textit{signed decomposition (of length at most $d$)} of $T=\{t_1 < \cdots < t_n\}$ is a decomposition of $T$ into a disjoint union of discrete intervals $I_i$, $i \in [k+1]$, for some integer $k$, $0 \leq k \leq d$, such that each interval is equipped with a sign $\sgn(I_i) \in \{ \pm 1 \}$ and for $i\in [k]$ holds $\sgn(I_i) = - \sgn (I_{i+1})$.
When such a decomposition is induced by a chamber $\sigma$ we call it \textit{signed $\sigma$-decomposition}.
\end{defn}

Note that if we were working with $-\sigma$ instead of $\sigma$, we would find the same set decomposition of $T$ but intervals $I_j$ would carry the opposite sign.
An example of a signed decomposition is depicted in \Cref{fig:signed-decomposition}. A more concrete example is given as follows.

\begin{figure}
	\begin{tikzpicture}
\draw[gray] (0.7,0) -- (12.3,0);
\draw[thick] (3.5,0.3) -- (3.5,-0.3)
				     (5.5,0.3) -- (5.5,-0.3)
				     (9.5,0.3) -- (9.5,-0.3)
foreach \s in{1,...,12}{
	(\s,0) node[fill,black,circle, inner sep=1.5pt]{}
	node[anchor=north]{$t_{\s}$}
};
\node[anchor=south] at (1,0) {$+$};
\node[anchor=south] at (2,0) {$+$};
\node[anchor=south] at (3,0) {$+$};
\node[anchor=south] at (4,0) {$-$};
\node[anchor=south] at (5,0) {$-$};
\node[anchor=south] at (6,0) {$+$};
\node[anchor=south] at (7,0) {$+$};
\node[anchor=south] at (8,0) {$+$};
\node[anchor=south] at (9,0) {$+$};
\node[anchor=south] at (10,0) {$-$};
\node[anchor=south] at (11,0) {$-$};
\node[anchor=south] at (12,0) {$-$};
\draw [decorate, decoration = {brace},thick] (0.7,0.6) --  (3.3,0.6);
\node at (2,1.1) {$I_1$};
\draw [decorate, decoration = {brace},thick] (3.7,0.6) --  (5.3,0.6);
\node at (4.5,1.1) {$I_2$};
\draw [decorate, decoration = {brace},thick] (5.7,0.6) --  (9.3,0.6);
\node at (7.5,1.1) {$I_3$};
\draw [decorate, decoration = {brace},thick] (9.7,0.6) --  (12.3,0.6);
\node at (11,1.1) {$I_4$};
\node at (-.1,0.25) {$\sgn(I_i)=$};
\end{tikzpicture}
	\caption{A signed decomposition of $12$ points on the real line, decomposed into $4$ discrete intervals $I_1, I_2, I_3, I_4$.}
	\label{fig:signed-decomposition}
\end{figure}
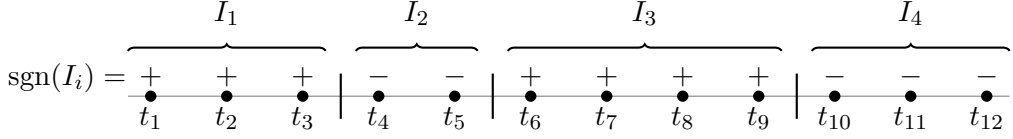

\begin{example}[unit directions]\label{ex:unit-direction}
	We consider polytopes $\poly(T)$ for $\xi = e_r$, where $e_0,$$\dots,$$e_d$ is the standard basis of $\R^{d+1}$, and $T = \{t_1,\dots,t_n\}$ such that $ t_1 < \dots < t_{m} < 0 < t_{m+1} < \dots < t_{n}$.
	For $r\in \{0,1\dots,d\}$, the curve $\psi / \inner{\psi,e_r}$
	is parametrised by  	
	\[
	\frac{\moment(t)}{\inner{\moment(t), e_r}} = \left(\frac{1}{t^r}, \frac{1}{t^{r-1}}, \dots, \frac{1}{t}, 1, t, t^2, \dots, t^{d-r} \right), \quad t\in \R.
	\]
	For $P_{e_r}(T)$ we have $T = I_1 \cup I_2 = \{t_1,\dots,t_m\} \cup \{t_{m+1}, \dots, t_n\}$ if $r$ is odd, and $T = I_1 = \{t_1,\dots,t_n\}$ if $r$ is even. 
\end{example}

\begin{rem}\label{induced cones}
The fact that 
$\sgn \circ \lambda_{\xi, S}(t)$, 
$t \in T$, depends on $\sigma$ only can be alternatively showed as follows. Recall that $\sgn \circ p_S$ is independent of $\xi$, we thus turn our attention to $\sgn \circ q_\xi$. Let $\xi \in \interior(\sigma)$ be fixed, hence the values
\begin{gather}\label{sign conditions}
\sgn \langle \xi, \nu_d(t) \rangle,
\hspace{1cm} t \in T
\end{gather}
are fixed.
We observe that the set of $\hat{\xi}$ having the same sign as $\xi$ in \eqref{sign conditions} are those $\hat{\xi}$ which lie on positive or negative side (depending on the fixed signs) of hyperplanes given by normal vectors $\nu_d(t)$, $t \in T$, and $\xi$ belongs to this set.
Of course, such a set is a chamber $\mathcal{C}$ in the hyperplane arrangement $\bigcup_{j=1}^n (\psi(\ell_j))^0$, equivalently in $\bigcup_{t \in T} (\nu_d(t))^0$, and it must be $\mathcal{C} = \sigma$, since $\xi \in \mathcal{C}$.
\end{rem}

On the other hand, for any signed $\sigma$-decomposition there exists a polynomial $q_\xi$, $\xi \in \h^*$, of degree at most $d$ such that its sign on $I_i$ is exactly $\sgn(I_i)$. 
More explicitly, such a polynomial can be constructed as follows. Let $s_i = \tfrac{1}{2} \left(\max(I_i) + \min(I_{i+1})\right)$. We can choose $q_\xi(t) = \prod_{i=1}^{k} (t - s_i)$, where $\xi$ is the vector of coefficients in the expansion of $q_\xi$, being elementary symmetric polynomials in $s_1,\dots,s_k$.
Now, viewing $q_\xi(t)$ as the contraction $\langle \xi, \nu_d(t) \rangle$, \Cref{induced cones} implies that $\xi$ belongs to a unique chamber $\sigma$. 
Therefore, the signed $\sigma$-decomposition induces the chamber $\sigma$. Similarly, a $\sigma$-decomposition induces the union $\sigma \cup \{ -\sigma \}$. \Cref{induced cones} shows that these correspondences are inverse to each other.
We thus proved the following.
\begin{thm}
There is a bijective correspondence between chambers $\sigma$ of the hyperplane arrangement $\bigcup_{j=1}^n (\psi(\ell_j))^0$ and signed $\sigma$-decompositions of $T$ of length at most $d$ induced by the above constructions. Furthermore, there is a bijective correspondence between $\sigma \cup \{ -\sigma \}$ for any chamber $\sigma$ and $\sigma$-decomposition of $T$.
\end{thm}
Observe that any signed $\sigma$-decomposition of $T$ of length at most $d$ is induced by a chamber. In particular, there are $2 \sum_{j=0}^d {n-1 \choose j}$ chambers, being the maximal number of chambers a central hyperplane arrangement on $n$ hyperplanes in a $(d+1)$-dimensional space can have.

\begin{corollary}\label{geometric Gale c}
In the setting of \Cref{geometric Gale}, $S$ corresponds to a facet if and only if there exists $c \in \{ \pm 1 \}$ such that for each $i\in[k+1]$ and for each $t \in I_i \backslash S$ we have $(\sgn \circ p_S) (t) = c \cdot \sgn(I_i)$.
\end{corollary}

\begin{rem}\label{proj pic Gale 2}
In terms of $\mathbb{S}^1$ from \Cref{proj pic Gale}, the $\sigma$-decomposition substitutes the $r$ points corresponding to infinities with respect to the chart $\xi$, and introduces a decomposition on $\mathbb{S}^1$ in an obvious way. However, because we work in a chart on $\mathbb{P}(W)$, we cannot say at this point, if points on $\mathbb{S}^1$ corresponding to $I_1$ and $I_{k+1}$ are really members of distinct arcs or if they belong together.
\end{rem}

\begin{rem}\label{P sigma}
When we are interested only in combinatorial properties of a Veronese polytope $P_\xi(T)$, we make use of the notation $P_\sigma(T)$, since the combinatorial type depends on the chamber $\sigma$ only. Fixing any $\xi \in \interior(\sigma)$, one recovers a geometric realization $P_\xi(T)$ as the section of $\sigma^\vee$ by the affine chart $\xi$.
\end{rem}

\section{Combinatorics of Veronese polytopes}
\label{sec:combinatorics}

This section gives multiple combinatorial characterizations of Veronese polytopes, resulting in a bijection between combinatorial types and  isomorphism classes of cyclically ordered sets with marked points.
These combinatorial descriptions are further used to identify combinatorial types which are realisable as Veronese polytopes, and give an explicit formula for the number of facets. 

Throughout this section, we assume that $T = \{t_1 < \dots < t_n\} \subset \R$ is an ordered set of cardinality $n$, and that its subsets are equipped with the inherited ordering. In general, $S \subset T$ is to be considered as a set of cardinality $d$.

\subsection{Combinatorial characterizations of facets}\label{sec:facets-line}
This subsection provides two equivalent combinatorial formulations of the generalised Gale evenness condition from \Cref{geometric Gale} (\Cref{geometric Gale c}), and includes an example of a Veronese polytope which is not cyclic.

\begin{defn}[$\sigma$-parity alternating sequence]\label{def:pa-sequences}
	Let $T$ be equipped with a signed $\sigma$-decomposition, inducing the partition $T = \bigcup_{j=1}^{k+1} I_j$. A subsequence $l_1< l_2 < \dots< l_{m}$ of $\{1<\dots<n\}$ (or $\{0<1<\dots<n\}$) is \emph{parity alternating} (p.a.) if $l_i$ and $l_{i+1}$ have distinct parities for all $i\in[m-1]$. The sequence is \emph{$\sigma$-parity alternating} ($\sigma$-p.a.) if for all $i\in[m-1]$ the following condition is satisfied:
	\[
		\sgn_\sigma(t_{l_i})  = \sgn_\sigma(t_{l_{i+1}}) 
		\text{ if and only if } l_i, l_{i+1} \text{ have different parities}.
	\]
\end{defn}

\begin{rem}
Observe that the function $\sgn_\sigma$ in \Cref{def:pa-sequences} is used only to determine if both $t_{l_i}$ and $t_{l_{i+1}}$ belong simultaneously either to $I_{\text{even}}^{\sigma}$ or to $I_{\text{odd}}^{\sigma}$, where
\begin{gather*}\label{internal}
	I_{\text{even}}^{\sigma} := \bigcup_{\substack{i=1 \\ i \text{ even}}}^{k+1} I_i
	\hspace{.2cm}
	\text{ and }
	\hspace{.2cm}
	I_{\text{odd}}^{\sigma} := \bigcup_{\substack{i=1 \\ i \text{ odd}}}^{k+1} I_i \ .
\end{gather*}
Furthermore, because $-\sigma$ induces the signed $(-\sigma)$-decomposition with the same set-theoretical partition $T = \bigcup_{j=1}^{k+1} I_j$ as $\sigma$, 
a sequence is $\sigma$-p.a. if and only if it is $(-\sigma)$-p.a.
Therefore, the definition of a $\sigma$-parity alternating sequence depends only on $\sigma \cup (- \sigma )$, or, equivalently, on the sets $I_{\text{even}}^{\sigma} = I_{\text{even}}^{-\sigma} $ and $I_{\text{odd}}^{\sigma} = I_{\text{odd}}^{-\sigma} $. 
\end{rem}

\begin{thm}\label{facets Veronese}\label{cor:bijection-facets-xi-pa}
	$S$ corresponds to a facet of $\polyc(T)$ if and only if $L   =  [n] \setminus \{j \mid t_j \in S\} = \{l_1 < \dots < l_{n-d}\}$ is a $\sigma$-parity alternating sequence.
\end{thm}

\begin{proof}
	Observe that for
	\begin{gather*}
		p_S(t) = \prod_{t_j \in S} (t - t_j),
	\end{gather*}
	the signs of $p_S(t_{l_i})$ and $p_S(t_{l_{i+1}})$ are equal if and only if $l_i$ and $l_{i+1}$ have distinct parities.\par
	
	If $S$ corresponds to a facet, then by \Cref{geometric Gale} the function $\sgn \circ \lambda_{\xi,S}$, depending on $\interior(\sigma)$ only, is constant on $T\setminus S$. Thus, we have
	 $\sgn_\sigma(t_{l_i}) = \sgn_\sigma(t_{l_{i+1}})$ if and only if $\sgn \circ p_S (t_{l_i}) = \sgn \circ p_S (t_{l_{i+1}})$.
	By the above observation, this is true if and only if $l_i$ and $l_{i+1}$ have distinct parities. This shows that $L$ is $\sigma$-parity alternating. \par
	
	To prove the other implication, let $L$ be a $\sigma$-parity alternating sequence. We start by showing that $p_S$ has constant sign on $L \cap I_i$, where $L = \bigcup_{i=1}^{k+1} L \cap I_i$.
	Cases when $| L \cap I_i | \leq 1$ are obvious.
	Let $| L \cap I_i | \geq 2$ and let $t_{l_r}, t_{l_{r+1}} \in L \cap I_i$ be two consecutive points.
	Since $L$ is $\sigma$-parity alternating, $l_r$ and $l_{r+1}$ have different parities, which by the above observation, is equivalent with
	$\sgn \circ p_S (t_{l_r}) = \sgn \circ p_S (t_{l_{r+1}})$, and thus $p_S$ has constant sign on $L \cap I_i$. \par
	Let $I_i$ and $I_{i+p}$, $p\geq1$, be intervals such that $| L \cap I_i | \neq 0$, $| L \cap I_{i+p} | \neq 0$ and such that for all $j=1,\ldots,p-1$ we have $| L \cap I_j | = 0$.
	 Let $t_{l_r} = \max I_i$ and $t_{l_{r+1}} = \min I_{i+p}$.
	 By \Cref{geometric Gale c} it remains to show that, up to a global sign $c \in \{ \pm 1 \}$, for each $t \in I_i \backslash S, i \in [k+1]$ we have $(\sgn \circ p_S) (t) = c \cdot \sgn(I_i)$.
	 Now, because $\sgn \circ p_S$ is constant on any interval $L \cap I_i$, we find that
	 $\sgn \circ p_S|_{I_i}$ and $\sgn \circ p_S|_{I_{i+p}}$ are equal if and only if $\sgn \circ p_S(t_{l_r})$ and $\sgn \circ p_S(t_{l_{r+1}})$ are equal,
	 and by the above observation this is equivalent with $l_r$ and $l_{r+1}$ having distinct parities.
	 Because $L$ is $\sigma$-parity alternating, the latter is equivalent with $\sgn_\sigma(t_{l_r})$ and $\sgn_\sigma(t_{l_{r+1}})$ being equal, which, by definition of $t_{l_r}$ and $t_{l_{r+1}}$ and by definition of $\sgn_\sigma$, is equivalent with $\sgn(I_i)$ and $\sgn(I_{i+p})$ being equal.
	 We showed that $\sgn \circ p_S|_{I_i}$ and $\sgn \circ p_S|_{I_{i+p}}$ are equal if and only if $\sgn(I_i)$ and $\sgn(I_{i+p})$ are equal, which proves the claim.
\end{proof}

Recall that for $\xi = (1,0,\dots,0)$, the Veronese polytope $P_\xi(T)$ is a cyclic polytope. In this case, a $\xi$-p.a. sequence is a parity alternating sequence, and \Cref{cor:bijection-facets-xi-pa} recovers the classical Gale evenness condition, which states that $l_{i+1} - l_{i} - 1$ is an even number. It is well-known that cyclic polytopes maximize the number of faces in each dimension. The following example describes a Veronese polytope which has strictly fewer facets than corresponding cyclic polytope.

\begin{example}\label{ex:facets}
	We consider the $4$-dimensional Veronese polytope $\poly(T)$ with $T = \{t_1,\dots,t_7\} = \{-3,-2,-1,1,2,3,4\}$ and $\xi = (0,-1,0,0)$. Then $T = I_1 \cup I_2$ with $I_1 = \{t_1,t_2,t_3\}$ and $I_2 = \{t_4,t_5,t_6,t_7\}$ according to the sign of $q_\xi(t_j) = \sum_{i=0}^4 \xi_i t_j^i = -t_j$. We visualize these points on the real line as follows:
	\begin{center}
		\begin{tikzpicture}[scale=0.8]
			\draw[gray] (-3.5, 0) -- (4.5, 0);
			\draw[thick] (0, -0.3) -- (0, 0.3);
			\draw foreach \s in{-3,-2,-1,1,2,3,4}{
				(\s,0) node[fill,black,circle, inner sep=1.3pt]{}
			};
			\node at (-5,0.35) {$\sgn(q_\xi(t_j)) = $};
			\node[anchor=north] at (-3,0) {$t_1$};
			\node[anchor=north] at (-2,0) {$t_2$};
			\node[anchor=north] at (-1,0) {$t_3$};
			\node[anchor=north] at (1,0) {$t_4$};
			\node[anchor=north] at (2,0) {$t_5$};
			\node[anchor=north] at (3,0) {$t_6$};
			\node[anchor=north] at (4,0) {$t_7$};
			\node[anchor=south] at (-3,0) {$+$};
			\node[anchor=south] at (-2,0) {$+$};
			\node[anchor=south] at (-1,0) {$+$};
			\node[anchor=south] at (1,0) {$-$};
			\node[anchor=south] at (2,0) {$-$};
			\node[anchor=south] at (3,0) {$-$};
			\node[anchor=south] at (4,0) {$-$};
		\end{tikzpicture}
	\end{center}
	
	In contrast to the cyclic polytope $C_4(T)$, which has 14 facets, the polytope $\poly(T)$ has $12$ facets.
	\Cref{fig:facets} shows all $12$ subsets $S\subset T$ corresponding to facets, i.e., $$\left\{\point{t_j} \mid t_j \in S\right\}$$ are the vertices of a facet.
	This can be checked by \Cref{facets Veronese}, since $L = [n] \setminus \{j \mid t_j \in S\}$ is $\sigma$-parity alternating. 
	For each $S$, \Cref{fig:facets} displays the functions $p_S(t) = \prod_{t_j \in S} (t - t_j)$ (in blue), thus illustrating that $\lambda_{\xi, S}(t) = \frac{p_S(t)}{q_\xi(t)}$ has constant sign on $T\setminus S$ (cf. \Cref{geometric Gale}).
	\Cref{fig:facets} also shows that each generating point is contained in at least one facet, hence all of these points are vertices.
	The red circles mark the points in $S_2 \sqcup S_3$, which will be described in \Cref{lem:facet-roots}. In \Cref{ex:not-neigborly-faets} we continue this example, showing that $\poly(T)$ has indeed precisely $12$ facets.

	\begin{table}[ht]
		\hspace{-2cm}
		\centering
		\begin{tabularx}{0.5\textwidth}{r l}
			\begin{minipage}[m]{7.52em}
				$L = \{t_1,t_2,t_3\}$ \\ $S = \{t_4,t_5,t_6,t_7\}$
			\end{minipage} &
			\multicolumn{1}{m{11.5em}}{
				\resizebox{11.5em}{3em}{
					\begin{tikzpicture}
						\draw[scale=1, domain=0.85:4.15, smooth, variable=\x, blue, very thick] plot ({\x}, {(\x-1)*(\x-2)*(\x-3)*(\x-4)});
						\draw[gray] (-3.5, 0) -- (4.5, 0);
						\draw[thick] (0, -0.3) -- (0, 0.3);
						\draw foreach \s in{-3,-2,-1,1,2,3,4}{
							(\s,0) node[fill,black,circle, inner sep=2.5pt]{}
						};
						\draw[color=red, very thick](1,0) circle (6pt);
						\draw[color=red, very thick](4,0) circle (6pt);
					\end{tikzpicture}
			}} 
			\\[2em]
			\begin{minipage}[m]{7.52em}
				$L = \{t_1,t_2,t_4\}$ \\
				$S = \{t_3,t_5,t_6,t_7\}$ 
			\end{minipage} &
			\multicolumn{1}{m{11.5em}}{
				\resizebox{11.5em}{3em}{
					\begin{tikzpicture}
						\draw[scale=1, domain=-1.1:4.2, smooth, variable=\x, blue, very thick] plot ({\x}, {1/15*(\x+1)*(\x-2)*(\x-3)*(\x-4)});
						\draw[gray] (-3.5, 0) -- (4.5, 0);
						\draw[thick] (0, -0.3) -- (0, 0.3);
						\draw foreach \s in{-3,-2,-1,1,2,3,4}{
							(\s,0) node[fill,black,circle, inner sep=2.5pt]{}
						};
						\draw[color=red, very thick](-1,0) circle (6pt);
						\draw[color=red, very thick](4,0) circle (6pt);
					\end{tikzpicture}
			}} 
			\\[2em]
			\begin{minipage}[m]{7.52em}
				$L = \{t_1,t_2,t_6\}$ \\
				$S = \{t_3,t_4,t_5,t_7\}$ 
			\end{minipage} &
			\multicolumn{1}{m{11.5em}}{
				\resizebox{11.5em}{3em}{
					\begin{tikzpicture}
						\draw[scale=1, domain=-1.3:4.2, smooth, variable=\x, blue, very thick] plot ({\x}, {1/11*(\x+1)*(\x-1)*(\x-2)*(\x-4)});
						\draw[gray] (-3.5, 0) -- (4.5, 0);
						\draw[thick] (0, -0.3) -- (0, 0.3);
						\draw foreach \s in{-3,-2,-1,1,2,3,4}{
							(\s,0) node[fill,black,circle, inner sep=2.5pt]{}
						};
						\draw[color=red, very thick](-1,0) circle (6pt);
						\draw[color=red, very thick](4,0) circle (6pt);
					\end{tikzpicture}
			}} 
			\\[2em]
				\begin{minipage}[m]{7.52em}
				$L = \{t_1,t_5,t_6\}$ \\
				$S = \{t_2,t_3,t_4,t_7\}$ 
			\end{minipage} &
			\multicolumn{1}{m{11.5em}}{
				\resizebox{11.5em}{3em}{
					\begin{tikzpicture}
						\draw[scale=1, domain=-2.4:4.2, smooth, variable=\x, blue, very thick] plot ({\x}, {1/30*(\x+2)*(\x+1)*(\x-1)*(\x-4)});
						\draw[gray] (-3.5, 0) -- (4.5, 0);
						\draw[thick] (0, -0.3) -- (0, 0.3);
						\draw foreach \s in{-3,-2,-1,1,2,3,4}{
							(\s,0) node[fill,black,circle, inner sep=2.5pt]{}
						};
						\draw[color=red, very thick](1,0) circle (6pt);
						\draw[color=red, very thick](4,0) circle (6pt);
					\end{tikzpicture}
			}} 
			\\[2em]
			\begin{minipage}[m]{7.52em}
				$L = \{t_2,t_3,t_5\}$ \\
				$S = \{t_1,t_4,t_6,t_7\}$ 
			\end{minipage} &
			\multicolumn{1}{m{11.5em}}{
				\resizebox{11.5em}{3.5em}{
					\begin{tikzpicture}
						\draw[scale=1, domain=-3.3:4.5, smooth, variable=\x, blue, very thick] plot ({\x}, {1/65*(\x+3)*(\x-1)*(\x-3)*(\x-4)});
						\draw[gray] (-3.5, 0) -- (4.5, 0);
						\draw[thick] (0, -0.3) -- (0, 0.3);
						\draw foreach \s in{-3,-2,-1,1,2,3,4}{
							(\s,0) node[fill,black,circle, inner sep=2.5pt]{}
						};
						\draw[color=red, very thick](-3,0) circle (6pt);
						\draw[color=red, very thick](1,0) circle (6pt);
					\end{tikzpicture}
			}} 
			\\[2em]
			\begin{minipage}[m]{7.52em}
				$L = \{t_2,t_3,t_7\}$ \\
				$S = \{t_1,t_4,t_5,t_6\}$ 
			\end{minipage} &
			\multicolumn{1}{m{11.5em}}{
				\resizebox{11.5em}{4em}{
					\begin{tikzpicture}
						\draw[scale=1, domain=-3.1:3.35, smooth, variable=\x, blue, very thick] plot ({\x}, {1/25*(\x+3)*(\x-1)*(\x-2)*(\x-3)});
						\draw[gray] (-3.5, 0) -- (4.5, 0);
						\draw[thick] (0, -0.3) -- (0, 0.3);
						\draw foreach \s in{-3,-2,-1,1,2,3,4}{
							(\s,0) node[fill,black,circle, inner sep=2.5pt]{}
						};
						\draw[color=red, very thick](-3,0) circle (6pt);
						\draw[color=red, very thick](1,0) circle (6pt);
					\end{tikzpicture}
			}} 
		\end{tabularx}	
		\hspace{1em}
		\begin{tabularx}{0.4\textwidth}{|r l}
			\begin{minipage}[m]{7.52em}
				$L = \{t_2,t_4,t_5\}$ \\
				$S = \{t_1,t_3,t_6,t_7\}$ 
			\end{minipage} &
			\multicolumn{1}{m{11.5em}}{
				\resizebox{11.5em}{4em}{
					\begin{tikzpicture}
						\draw[scale=1, domain=-3.3:4.4, smooth, variable=\x, blue, very thick] plot ({\x}, {1/30*(\x+3)*(\x+1)*(\x-3)*(\x-4)});
						\draw[gray] (-3.5, 0) -- (4.5, 0);
						\draw[thick] (0, -0.3) -- (0, 0.3);
						\draw foreach \s in{-3,-2,-1,1,2,3,4}{
							(\s,0) node[fill,black,circle, inner sep=2.5pt]{}
						};
						\draw[color=red, very thick](-3,0) circle (6pt);
						\draw[color=red, very thick](-1,0) circle (6pt);
					\end{tikzpicture}
			}} 
			\\[2em]
			\begin{minipage}[m]{7.52em}
				$L = \{t_2,t_4,t_7\}$ \\
				$S = \{t_1,t_3,t_5,t_6\}$ 
			\end{minipage} &
			\multicolumn{1}{m{11.5em}}{
				\resizebox{11.5em}{4em}{
					\begin{tikzpicture}
						\draw[scale=1, domain=-3.2:3.4, smooth, variable=\x, blue, very thick] plot ({\x}, {1/15*(\x+3)*(\x+1)*(\x-2)*(\x-3)});
						\draw[gray] (-3.5, 0) -- (4.5, 0);
						\draw[thick] (0, -0.3) -- (0, 0.3);
						\draw foreach \s in{-3,-2,-1,1,2,3,4}{
							(\s,0) node[fill,black,circle, inner sep=2.5pt]{}
						};
						\draw[color=red, very thick](-3,0) circle (6pt);
						\draw[color=red, very thick](-1,0) circle (6pt);
					\end{tikzpicture}
			}} 
			\\[2em]
				\begin{minipage}[m]{7.52em}
				$L = \{t_2,t_6,t_7\}$ \\
				$S = \{t_1,t_3,t_4,t_5\}$ 
			\end{minipage} &
			\multicolumn{1}{m{11.5em}}{
				\resizebox{11.5em}{4em}{
					\begin{tikzpicture}
						\draw[scale=1, domain=-3.15:2.3, smooth, variable=\x, blue, very thick] plot ({\x}, {1/8*(\x+3)*(\x+1)*(\x-1)*(\x-2)});
						\draw[gray] (-3.5, 0) -- (4.5, 0);
						\draw[thick] (0, -0.3) -- (0, 0.3);
						\draw foreach \s in{-3,-2,-1,1,2,3,4}{
							(\s,0) node[fill,black,circle, inner sep=2.5pt]{}
						};
						\draw[color=red, very thick](-3,0) circle (6pt);
						\draw[color=red, very thick](-1,0) circle (6pt);
					\end{tikzpicture}
			}} 
			\\[2em]
			\begin{minipage}[m]{7.52em}
				$L = \{t_3,t_5,t_6\}$ \\
				$S = \{t_1,t_2,t_4,t_7\}$ 
			\end{minipage} &
			\multicolumn{1}{m{11.5em}}{
				\resizebox{11.5em}{3em}{
					\begin{tikzpicture}
						\draw[scale=1, domain=-3.5:4.2, smooth, variable=\x, blue, very thick] plot ({\x}, {1/40*(\x+3)*(\x+2)*(\x-1)*(\x-4)});
						\draw[gray] (-3.5, 0) -- (4.5, 0);
						\draw[thick] (0, -0.3) -- (0, 0.3);
						\draw foreach \s in{-3,-2,-1,1,2,3,4}{
							(\s,0) node[fill,black,circle, inner sep=2.5pt]{}
						};
						\draw[color=red, very thick](1,0) circle (6pt);
						\draw[color=red, very thick](4,0) circle (6pt);
					\end{tikzpicture}
			}} 
			\\[2em]
			\begin{minipage}[m]{7.52em}
				$L = \{t_4,t_5,t_6\}$ \\
				$S = \{t_1,t_2,t_3,t_7\}$ 
			\end{minipage} &
			\multicolumn{1}{m{11.5em}}{
				\resizebox{11.5em}{4em}{
					\begin{tikzpicture}
						\draw[scale=1, domain=-3.5:4.15, smooth, variable=\x, blue, very thick] plot ({\x}, {1/60*(\x+3)*(\x+2)*(\x+1)*(\x-4)});
						\draw[gray] (-3.5, 0) -- (4.5, 0);
						\draw[thick] (0, -0.3) -- (0, 0.3);
						\draw foreach \s in{-3,-2,-1,1,2,3,4}{
							(\s,0) node[fill,black,circle, inner sep=2.5pt]{}
						};
						\draw[color=red, very thick](-1,0) circle (6pt);
						\draw[color=red, very thick](4,0) circle (6pt);
					\end{tikzpicture}
			}} 
			\\[2em]
			\begin{minipage}[m]{7.52em}
				$L = \{t_5,t_6,t_7\}$ \\
				$S = \{t_1,t_2,t_3,t_4\}$ 
			\end{minipage} &
			\multicolumn{1}{m{11.5em}}{
				\resizebox{11.5em}{3em}{
					\begin{tikzpicture}
						\draw[scale=1, domain=-3.4:1.2, smooth, variable=\x, blue, very thick] plot ({\x}, {10/65*(\x+3)*(\x+2)*(\x+1)*(\x-1)});
						\draw[gray] (-3.5, 0) -- (4.5, 0);
						\draw[thick] (0, -0.3) -- (0, 0.3);
						\draw foreach \s in{-3,-2,-1,1,2,3,4}{
							(\s,0) node[fill,black,circle, inner sep=2.5pt]{}
						};
						\draw[color=red, very thick](-3,0) circle (6pt);
						\draw[color=red, very thick](1,0) circle (6pt);
					\end{tikzpicture}
			}} 
		\end{tabularx}
	
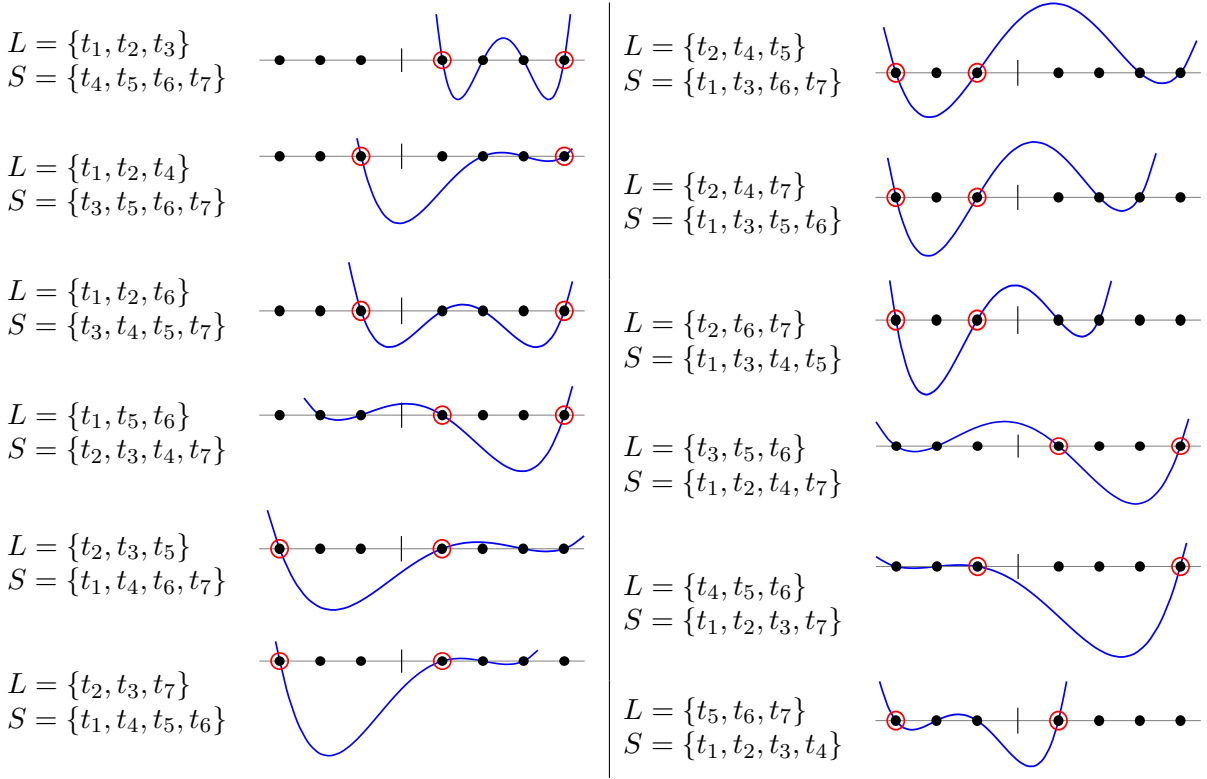
\captionof{figure}{The $12$ facets of the polytope from \Cref{ex:facets}.} 
	\label{fig:facets}
	\end{table}
\end{example}

A $\sigma$-decomposition of $T$ provides the partition $T = \bigcup_{j=1}^{k+1} I_j$, which invites to relabel elements of $T$ as follows. We denote $n_j = |I_j|, j \in [k+1]$, and index elements of $I_j$ uniquely as $t_{j,i}$, $i\in[n_j]$, by requiring that $t_{j,1} < \cdots < t_{j,n_j}$, as illustrated in \Cref{fig:dividers}. A \emph{consecutive pair} in $I_j$ is a pair of the form $\{t_{j,i},t_{j,i+1}\}$.

	\begin{figure}[b]
	\begin{tikzpicture}
	  [
	scale=0.7,
	point/.style={inner sep=1.2pt,circle,draw=black,fill=black,thick,scale=1.1},
	]
	\node[point, label=below:{$t_{1,1}$}] at (1,0) {};
	\node[point, label=below:{$t_{1,2}$}] at (2,0) {};
	\node at (3,0) {$\dots$};
	\node[point, label=below:{$t_{1,n_1}$}] at (4,0) {};
	\node at (4.5,0) {$\vert$};
	\node[point, label=below:{$\ t_{2,1}$}] at (5,0) {};
	\node at (6.1,0) {$\dots$};
	\node at (7,0) {$\dots$};
	\node[point, label=below:{\small $\! \! \! \! \! t_{j\text{-}1,n_{j\text{-}1}}$}] at (8,0) {};
	\node at (8.5,0) {$\vert$};
	\node[point, label=below:{$ \ t_{j,1}$}] at (9,0) {};
	\node at (10,0) {$\dots$};
	\node[point, label=below:{$\! \! t_{j,n_j}$}] at (11,0) {};
	\node at (11.5,0) {$\vert$};
	\node[point, label=below:{$\quad t_{j+1,1}$}] at (12,0) {};
	\node at (13.1,0) {$\dots$};
	\node at (14,0) {$\dots$};
	\node[point, label=below:{$\! \! \!  t_{k,n_k}$}] at (15,0) {};
	\node at (15.5,0) {$\vert$};
	\node[point, label=below:{$\quad \  t_{k+1,1}$}] at (16,0) {};
	\node at (17,0) {$\dots$};
	\node[point, label=below:{$\qquad t_{k+1,n_{k+1}}$}] at (18,0) {};
	\draw [decorate, decoration = {brace},thick] (1,0.5) --  (4,0.5);
	\node at (2.5,1.1) {$I_1$};
	\draw [decorate, decoration = {brace},thick] (9,0.5) --  (11,0.5);
	\node at (10,1.1) {$I_j$};
	\draw [decorate, decoration = {brace},thick] (16,0.5) --  (18,0.5);
	\node at (17,1.1) {$I_{k+1}$};
\end{tikzpicture}
	\caption{The indexing of $T$ for \Cref{LEMMA} and \Cref{lem:facet-roots}.
	}
	\label{fig:dividers}
\end{figure}
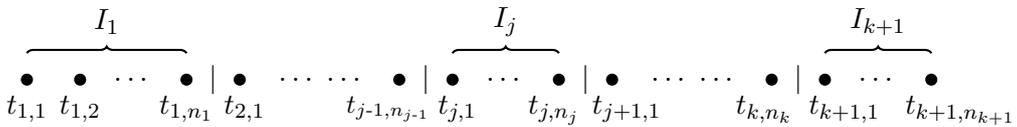

\begin{lemma}\label{LEMMA}
	Let $S$ correspond to a facet, $J \subseteq S$ be a maximal (discrete) interval with respect to inclusion, and $D = \{j\in [k] \mid J \cap \{t_{j,n_j},t_{j+1,1} \} \neq \emptyset\}$.
	Then there exists a unique decomposition $J = J_1 \sqcup J_2 \sqcup J_3$ satisfying
	\begin{enumerate}[label={\textup{(}J\arabic*\textup{)}},ref={J\arabic*}]
		\item $J_1 = \{J^j_1 \mid j \in D\}$ consists of $|D|$ distinct elements such that $J^j_1 \in \{t_{j,n_j}, t_{j+1,1}\}$ for each $j\in D$,  \label{eq:J1} 
		\item $J_2 \subseteq \{t_{1,1}, t_{k+1,n_{k+1}}\}$, \label{eq:J2}
		\item $J_3$ consists of non-intersecting consecutive pairs. \label{eq:J3}
	\end{enumerate}
\end{lemma}
\begin{proof}
Under the condition $t_{1,1} \notin J$, we first define $J_1, J_2$ and $J_3$, then find conditions under which they satisfy \eqref{eq:J1}, \eqref{eq:J2}, \eqref{eq:J3}, and then show that these conditions are fulfilled. Afterwards, we treat the case $t_{1,1} \in J$.
It can be checked that the constructed decomposition is the unique decomposition of $J$ satisfying \eqref{eq:J1}, \eqref{eq:J2}, \eqref{eq:J3}.

Note that $D$ is a discrete interval, and so $D = \{j_1,\dots,j_r\}$ for some $j_1 \in [k]$, where $r = |D|$ and for $i \in [r-1]$ we have $j_{i+1} = j_i +1$, and $j_{r} \leq k$.
We begin by defining $J_1 = \{J_1^{j_1},\dots,j_1^{j_r}\}$ inductively for $j \in D$ in increasing order by
\begin{equation*}
\begin{cases}
	\text{if }  |J \cap I_j \setminus \{J_1^{j-1}\}| \text{ is odd} : & J_1^j :=  t_{j, n_{j}}\\
	\text{if } |J \cap I_j \setminus \{J_1^{j-1}\}| \text{ is even} : & J_1^j := t_{j+1,1},
\end{cases}
\end{equation*}
where $\{J_1^{j_0}\} := \emptyset$ and $I_j$, $j \in [k+1]$, are discrete intervals induced by the signed $\sigma$-decomposition.
The set $J_2$ is defined by
\[
J_2 = \begin{cases}
	\{t_{k+1,n_{k+1}}\} & \text{if } t_{k+1,n_{k+1}} \in J \text{ and } |I_{k+1} \setminus J_1| \text{ odd} \\
	\emptyset & \text{ otherwise},
\end{cases}
\]
which clearly satisfies \eqref{eq:J2}. Finally, let $J_3 := J \setminus (J_1 \cup J_2)$.

The set $J_1$ satisfies $|J_1| = |D|$ and $J_1^j \in \{t_{j,n_j}, t_{j+1,1}\}$ for $j \in D$ by construction. To show \eqref{eq:J1}, it thus remains to show that $J_1 \subseteq J$. 
Note that if $\{t_{j_i,n_{j_i}}, t_{j_{i}+1,1}\} \subseteq J$, then $J^{j_i}_1 \in J_1$.
Since $J$ is an interval, this holds in general for $j_2,\dots,j_{r-1}$.
Now, for $j=j_1$, 
if $\{t_{j_1,n_{j_1}}, t_{j_{1}+1,1}\} \not\subseteq J$, then either $t_{j_1,n_{j_1}} \not \in J$, in which case $|J \cap I_{j_1} \setminus \{J_1^{j_0}\}| = 0$ is even, and $J_1^{j_1} = t_{j_1 + 1,1} \in J$, or $t_{j_1 + 1, n_{j_1 + 1}} \not \in J$. In the latter case we necessarily have $r = 1$ and $j_1 = j_r$, a case covered by the following condition 1).
For $j=j_r$, if $\{t_{j_r,n_{j_r}}, t_{j_{r}+1,1}\} \not\subseteq J$, then $t_{j_{r}+1,1} \not \in J$, since the other case is covered above ($r=1$).
It follows that $J_1^{j_r} \in J$, i.e., $J_1^{j_r} = t_{j_r,n_{j_r}}$, if and only if $|J \cap I_{j_r} \setminus \{J_1^{j_r - 1}\}|$ is odd.
It thus remains to show the following condition for $j = j_r$:
\begin{enumerate}
	\item[1)] If  $J \cap I_{j_r + 1} = \emptyset$, then $|J \cap I_{j_r} \setminus \{J_1^{j_r - 1}\}|$ is odd.
\end{enumerate}
To prove that $J_3$ satisfies \eqref{eq:J3}, we need to show that $|J_3 \cap I_j|$ is even for all $j \in [k+1]$. 
For $D \neq \emptyset$, the definition of $J_1$ implies that this is satisfied for $j \neq j_r + 1$, and always satisfied if 
$J \cap I_{j_{r}+1} = \emptyset$. Moreover, if $t_{k+1,n_{k+1}} \in J$ then the definition of $J_2$ implies that $|J_3 \cap I_{k+1}|$ is even.  
Thus, for $D \neq \emptyset$, it remains to show
\begin{enumerate}
	\item[2)]  if $t_{k+1,n_{k+1}} \not \in J$ and $J \cap I_{j_{r}+1} \neq \emptyset$ then $|J \cap I_{j_{r+1}} \setminus \{J_1^{j_r}\}|$ is even.
\end{enumerate}
If $D = \emptyset$, then $J_1 = \emptyset$. If $t_{k+1,n_{k+1}} \in J$ then $|J_3|$ is even by definition of $J_2$. Otherwise, we need to show
\begin{enumerate}
	\item[3)] if $D = \emptyset$ and $t_{k+1,n_{k+1}} \not \in J$, then $|J| = |J_3|$ is even.
\end{enumerate} 

Now we show conditions 1), 2) and 3). In all three cases we have
$t_{1,1} \not \in J$ and $t_{k+1,n_{k+1}}\not\in J$, so
\begin{gather*}
a = \max\{t \in T \mid \ \forall u \in J : t < u \}, 
\quad
b = \min\{t \in T \mid \ \forall u \in J : t > u \}, 
\end{gather*}
exist.
The maximality of $J$ implies that $a$ and $b$ are contained in $T\setminus S$, which, together with their definition, shows that they are consecutive elements in the $\sigma$-parity alternating sequence $T \setminus S$, and so $a = t_{l_u}$ and $b = t_{l_{u+1}}$ for some index $u \in [n-d-1]$ (cf. \Cref{facets Veronese}).
It follows that $|J| = l_{u+1} - l_u -1$ is even if and only if $\sgn_\sigma(t_{l_{u}})) = \sgn_\sigma(t_{l_{u+1}})$.
In case 3), $t_{l_u}, t_{l_{u+1}}$ are contained in the same interval $I_i$ (for some $i \in [k+1]$), thus proving 3).
If $D \neq \emptyset$, we have $t_{l_u} \in I_{j_1}$ and $t_{l_{u+1}} \in I_{j_r +1}$, and thus $\sgn_\sigma(t_{l_{u+1}}) = (-1)^{j_r + 1 - j_1} \sgn_\sigma(t_{l_{u}}) = (-1)^{r} \sgn_\sigma(t_{l_{u}})$. Therefore, $|J|$ is even if and only if $r$ is even.
For case 1) holds
\begin{align*}
	|J| &= 
	|\{ t \in J \mid t \leq J_1^{j_{r-1}} \}| + |\{ t \in J \mid t > J_1^{j_{r-1}} \}| \\
	&= 
	(r-1) + \sum_{j = j_1}^{j_{r - 1}} 
	| J \cap I_j \setminus \{J_1^{j_1},\dots,J_1^{j_{r-1}}\}| + |J \cap I_{j_r} \setminus \{J_1^{j_r - 1}\}|,
\end{align*}
where the second sum is a sum of even numbers by construction. Therefore, $|J \cap I_{j_r} \setminus \{J_1^{j_r - 1}\}|$ is odd, hence proving 1).
For case 2) holds
\begin{align*}
	|J| &= 
	|\{ t \in J \mid t \leq J_1^{j_{r}} \}| + |\{ t \in J \mid t > J_1^{j_{r}} \}| \\
	&= 
	r + \sum_{j = j_1}^{j_{r}} 
	| J \cap I_j \setminus J_1| + |J \cap I_{j_r +1} \setminus \{J_1^{j_r}\}|,
\end{align*}
where the second sum is an even number by construction. Therefore, $|J \cap I_{j_r+1} \setminus \{J_1^{j_r}\}|$ is even, hence proving 2). \par

To address $J$ such that $t_{1,1} \in J$, we define $J_1 = \{J_1^{1},\dots,J_1^{r}\}$ inductively for $j \in D = \{1,\dots, r\}$ in decreasing order by
\begin{equation*}
	\begin{cases}
		\text{if }  |J \cap I_{j+1} \setminus \{J_1^{j+1}\}| \text{ is odd} : & J_1^j :=  t_{j+1, 1}\\
		\text{if } |J \cap I_{j+1} \setminus \{J_1^{j+1}\}| \text{ is even} : & J_1^j := t_{j,n_j},
	\end{cases}
\end{equation*}
where $\{J_1^{r+1}\} := \emptyset$. 
	As before, $J_1$ satisfies $|J_1| = |D|$ and $J_1^j \in \{t_{j,n_j}, t_{j+1,1}\}$ for $j \in D$, so it remains to show that $J_1 \subseteq J$. Now, $\{t_{j,n_{j}}, t_{j+1,1}\} \subseteq J$ for $j=1,\dots,r-1$, so $J^{j}_1 \in J_1$.
	For $j=r$, 
	if $\{t_{r,n_{r}}, t_{r+1,1}\} \not\subseteq J$, then $t_{r+1,1} \not \in J$, in which case $|J \cap I_{r+1} \setminus \{J_1^{r+1}\}| = 0$ is even, and $J_1^{r} = t_{r,n_r} \in J$. Thus, $J_1$ satisfies \eqref{eq:J1}.

For defining $J_2$, note that $t_{k+1,n_{k+1}} \not \in J$, since $J$ is an interval of length $|J| \leq |S| = d < n$. We thus define $J_2$ by
\[
J_2 = \begin{cases}
	\{t_{1,n_{1}}\} & \text{if } |I_{1} \setminus J_1| \text{ is odd} \\
	\emptyset & \text{ otherwise},
\end{cases}
\]
which clearly satisfies \eqref{eq:J2}. Finally, for $J_3 := J \setminus (J_1 \cup J_2)$ holds that $|J_3 \cap I_j|$, $j \in [k+1]$, is even by construction of $J_1$ and $J_2$, and hence $J_3$ satisfies \eqref{eq:J3}.

\end{proof}

\begin{thm}\label{lem:facet-roots}
Let $T = \bigcup_{j=1}^{k+1} I_j$ be the partition induced by the signed $\sigma$-decomposition corresponding to a Veronese polytope $P_\sigma(T)$.
	A subset $S\subset T$ of cardinality $d$ corresponds to a facet of $\polyc(T)$ if and only if there exists a decomposition $S = S_1 \sqcup S_2 \sqcup S_3$, such that
	\begin{enumerate}[label={\textup{(}S\arabic*\textup{)}}, ref={S\arabic*}]
		\item $S_1 = \{S^1_1,\dots,S^k_1\}$ consists of $k$ distinct elements such that $S^j_1 \in \{t_{j,n_j}, t_{j+1,1}\}$ for each $j\in [k]$, \label{condition-1}
		\item $S_2 \in \{ \emptyset, \{t_{1,1}\}, \{t_{k+1,n_{k+1}}\},  \{t_{1,1}, t_{k+1,n_{k+1}}\} \}$ and $d - k - | S_2 |$ is even, \label{condition-2}
		\item $S_3$ consists of the remaining $d-k-|S_2|$ points, which occur as $\frac{d - k - |S_2|}{2}$ consecutive pairs.
		\label{condition-3}
	\end{enumerate}
	Furthermore, this decomposition is unique for every facet.
\end{thm}

\begin{proof}
	We first show that $S = S_1 \sqcup S_2 \sqcup S_3$ satisfying \eqref{condition-1}--\eqref{condition-3} forms a facet.
	Given $S_1\sqcup S_2 \sqcup S_3$, condition \eqref{condition-3} ensures that $p_S$
	(as defined in \eqref{eq:polynomials})
	has constant sign on each non-empty $I_j  \setminus S$. 
	Conditions \eqref{condition-1} and \eqref{condition-3} imply that 
	if $I_j \setminus S$ and $I_{j'} \setminus S$, $j<j'$, are non-empty, then $\sgn \circ p_S|_{I_j \setminus S} = (-1)^{j'-j} \sgn \circ p_S|_{I_j' \setminus S}$. Therefore, \Cref{geometric Gale c} implies that $S$ corresponds to a facet.

	Conversely, suppose that $S \subset T$ corresponds to a facet.
	By \Cref{LEMMA}, each maximal interval $J \subset S$ has a unique decomposition $J = J_1 \sqcup J_2 \sqcup J_3$ satisfying  \eqref{eq:J1}, \eqref{eq:J2}, \eqref{eq:J3}. 
	We define $S_i = \bigcup_{J} J_i$, $i\in [3]$, where the union ranges over all maximal intervals of $S$. Note that these unions are disjoint.
	
	To prove that $S_1$ satisfies \eqref{condition-1} it suffices to show that $S_1 \cap \{ t_{j,n_j}, t_{j+1,1} \} \neq \emptyset$ for every $j \in [k]$.
	Observe that $t_{j,n_j}$ or $t_{j+1,1}$ is a root of $p_S$, since otherwise, for some $j$ we would have
	$\sgn \circ p_S (t_{j,n_j}) = \sgn \circ p_S (t_{j+1,1})$, which contradicts $S$ being a facet.
	Therefore, $t_{j,n_j} \in S$ or $t_{j+1,1} \in S$ is contained in a unique maximal interval $J \subseteq S$. Thus, $t_{j,n_j} \in J_1$ or $t_{j+1,n_{j+1}} \in J_1$, and $J_1 \subseteq S_1$.
	
	By definition, $S_3$ consists of $|S_3|/2$ consecutive non-intersecting pairs, where
	\[
	|S_3| = |S| - |S_1| - |S_2| = d - k - |S_2|,
	\]
 and thus \eqref{condition-3} holds.
In particular, $d-k-|S_2|$ is even, and by definition of $J_2$ we have $S_2 \subset \{ t_{1,1}, t_{k+1,n_{k+1}} \}$, which implies \eqref{condition-2}.
\end{proof}

\subsection{Circular facet condition}\label{sec:cyclic-gale}
This section finds that conditions \eqref{condition-1}-\eqref{condition-3} of \Cref{lem:facet-roots} behave more naturally when the ground set $T$ is equipped with a cyclic order.
In turn, it allows us to find many combinatorial isomorphisms of Veronese polytopes, as well as to reinterpret them as cyclically ordered sets with marked points, called dividers.
Cyclically ordered sets were originally introduced by Huntington \cite{huntington16_setindependentpostulates,huntington35_interrelationsfour}.
\begin{defn}[Cyclic Order]
	A \emph{cyclic order} on a non-empty set $\mathring{T}$ is a ternary relation $C \subset \mathring{T} \times \mathring{T} \times \mathring{T}$ satisfying
	\begin{enumerate}
		\item if $(a,b,c) \in C$, then $(b,c,a) \in C$
		\item if $(a,b,c) \in C$, then $(c,b,a)$ does not belong to $C$
		\item if $(a,b,c) \in C$ and $(a,c,d) \in C$, then $(a,b,d) \in C$
		\item if $a,b,c \in \mathring{T}$ are mutually distinct, then either $(a,b,c) \in C$ or $(c,b,a) \in C$
	\end{enumerate}
If $(a,b,c) \in C$ are pairwise distinct, then we say that $b$ \emph{follows after} $a$ and \emph{lies before} $c$ with respect to the cyclic order $C$. For $a,b \in \cT$, $a\neq b$, $b$ \textit{follows consecutively after} $a$ if for all $x \in \cT \backslash \{a,b\}$ we have $(b,x,a) \in C$.
Finally, $a$ and $b$ form a \textit{consecutive pair} if either $b$ follows consecutively after $a$, or $a$ follows consecutively after $b$.
A \emph{divider} is a subset $\{a,b\} \subset \cT$ such that $a$ and $b$ form a consecutive pair.
\end{defn}

\begin{defn}\label{circular Gale}
Let $d$ be a positive integer and $\cT$ a set equipped with a cyclic order such that $|\cT| \geq d+1$. Let
$\cD = \left\{ \{a_1,b_1\}, \ldots, \{a_l,b_l\} \right\}$
be a set of dividers such that $|\cD|=l$, $0 \leq l \leq d$, and $l$ and $d$ have the same parities. The pair $(\mathring{T}, \cD)$ is called a \textit{circular composition} in dimension $d$.
A subset $\mathring{S} \subset \mathring{T}$ of cardinality $d$ satisfies the \emph{circular facet condition} (with respect to $\cD$) if there exists a partition $\cS = \cS_1 \sqcup \cS_2$ such that
	\begin{enumerate}[label={($\mathring{\text{S}}$\arabic*)}, ref={$\mathring{\text{S}}$\arabic*}]
		\item \label{circgale:dividers} $\cS_1 = \{\cS_1^1,\dots,\cS_1^l\}$ consists of $l$ distinct elements such that $\cS_1^j \in \{a_j,b_j\}$ for $j\in[l]$,
		\item  $S_2$ consists of the remaining $d-l$ points, which occur as $\frac{d-l}{2}$ consecutive pairs.
		\label{circgale:pairs}
	\end{enumerate}
	The set of all subsets $\cS \subset \cT$ of cardinality $d$ satisfying the circular facet condition is denoted by $\cfacets(\mathring{T})$.
	An \emph{isomorphism of circular compositions} $(\cT,\cD)$ and $(\cT',\cD')$ is a map $\kappa: \mathring{T} \to \mathring{T}'$ which induces a bijection between $\cfacets(\mathring{T})$ and $\cfacets[\cD'](\mathring{T}')$.
\end{defn}

\begin{figure}[b]
	\begin{tikzpicture}[scale=0.8]
	\def \n {18}
	\def \radius {2}
	\def \small {0.6pt}
	\def \big {1.6pt}
	\draw[color=gray] circle(\radius);
	\draw[color=black]
	(-360/\n*1:-\radius) node[fill,circle, inner sep=\big]{}
	node[anchor=-360/\n*1]{$P_{1,1}$}
	(-360/\n*2:-\radius) node[fill,circle, inner sep=\big]{}
	node[anchor=-360/\n*2]{$P_{1,2}$}
	(-360/\n*2.7:-\radius) node[fill,circle, inner sep=\small]{}
	(-360/\n*3:-\radius) node[fill,circle, inner sep=\small]{}
	(-360/\n*3.3:-\radius) node[fill,circle, inner sep=\small]{}
	(-360/\n*4:-\radius) node[fill,circle, inner sep=\big]{}
	node[anchor=-360/\n*2.5]{$P_{1,m_1}$}
	(-360/\n*5:-\radius) node[fill,circle, inner sep=\big]{}
	node[anchor=-360/\n*6]{$P_{2,1}$}
	(-360/\n*5.7:-\radius) node[fill,circle, inner sep=\small]{}
	(-360/\n*6:-\radius) node[fill,circle, inner sep=\small]{}
	(-360/\n*6.3:-\radius) node[fill,circle, inner sep=\small]{}
	(-360/\n*6.7:-\radius) node[fill,circle, inner sep=\small]{}
	(-360/\n*7:-\radius) node[fill,circle, inner sep=\small]{}
	(-360/\n*7.3:-\radius) node[fill,circle, inner sep=\small]{}
	(-360/\n*8:-\radius) node[fill,circle, inner sep=\big]{}
	node[anchor=-360/\n*8.5]{$P_{j-1,m_{j-1}}$}
	(-360/\n*9:-\radius) node[fill,circle, inner sep=\big]{}
	node[anchor=-360/\n*9.5]{$P_{j,1}$}
	(-360/\n*9.7:-\radius) node[fill,circle, inner sep=\small]{}
	(-360/\n*10:-\radius) node[fill,circle, inner sep=\small]{}
	(-360/\n*10.3:-\radius) node[fill,circle, inner sep=\small]{}
	(-360/\n*11:-\radius) node[fill,circle, inner sep=\big]{}
	node[anchor=-360/\n*10]{$P_{j,m_j}$}
	(-360/\n*12:-\radius) node[fill,circle, inner sep=\big]{}
	node[anchor=-360/\n*11]{$P_{j+1,1}$}
	(-360/\n*12.7:-\radius) node[fill,circle, inner sep=\small]{}
	(-360/\n*13:-\radius) node[fill,circle, inner sep=\small]{}
	(-360/\n*13.3:-\radius) node[fill,circle, inner sep=\small]{}
	(-360/\n*13.7:-\radius) node[fill,circle, inner sep=\small]{}
	(-360/\n*14:-\radius) node[fill,circle, inner sep=\small]{}
	(-360/\n*14.3:-\radius) node[fill,circle, inner sep=\small]{}
	(-360/\n*15:-\radius) node[fill,circle, inner sep=\big]{}
	node[anchor=-360/\n*12]{$P_{l-1,m_{l-1}}$}
	(-360/\n*16:-\radius) node[fill,circle, inner sep=\big]{}
	node[anchor=-360/\n*17]{$P_{l,1}$}
	(-360/\n*16.7:-\radius) node[fill,circle, inner sep=\small]{}
	(-360/\n*17:-\radius) node[fill,circle, inner sep=\small]{}
	(-360/\n*17.3:-\radius) node[fill,circle, inner sep=\small]{}
	(-360/\n*18:-\radius) node[fill,circle, inner sep=\big]{}
	node[anchor=-360/\n*18]{$P_{l,m_{l}}$}
	;
	\draw (-360/\n*0.5:-\radius-0.4) -- (-360/\n*0.5:-\radius+0.4);
	\draw (-360/\n*4.5:-\radius-0.4) -- (-360/\n*4.5:-\radius+0.4);
	\draw (-360/\n*8.5:-\radius-0.4) -- (-360/\n*8.5:-\radius+0.4);
	\draw (-360/\n*11.5:-\radius-0.4) -- (-360/\n*11.5:-\radius+0.4);
	\draw (-360/\n*15.5:-\radius-0.4) -- (-360/\n*15.5:-\radius+0.4);
	\draw[decorate, decoration = {brace, amplitude=18pt}] (-3.2,1.1) to (-0.5,3);
	\node at (-2.7,3) {$\mathring{I}_1$};
	\draw[decorate, decoration = {brace, amplitude=10pt}] (3.2,0.2) to (2.9,-2);
	\node at (3.9,-1) {$\mathring{I}_j$};
	\draw[decorate, decoration = {brace, amplitude=12pt}] (-2.4,-1.9) to (-3.5,0.3);
	\node at (-4,-1.3) {$\mathring{I}_{l}$};
\end{tikzpicture}
	\caption{The circular composition $\cT = \bigcup_{j=1}^l \cI_j$ induced by the signed $\sigma$-decomposition $T=\bigcup_{j=1}^{k+1} I_j$ from \Cref{fig:dividers}.}
	\label{fig:circle-ordering}
\end{figure}
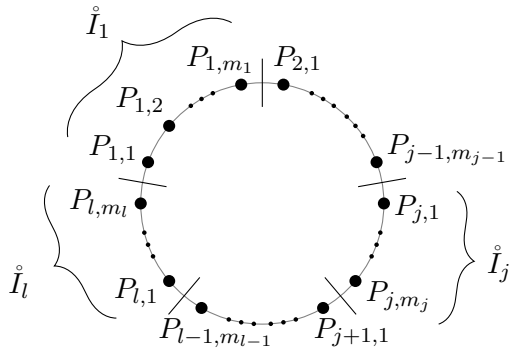

Now we construct a circular composition $(\mathring{T}, \cD)$ and a structure map $\tau: T \to \cT$ from a Veronese $d$-polytope $P_\sigma(T)$.
Recall that the associated cone $\sigma$ induces a signed $\sigma$-decomposition, in particular $T = \bigcup_{j=1}^{k+1} I_j$ for some $k \in \{ 0,\ldots,d \}$. 
Let $\cT$ be a set equipped with a cyclic order such that $|\cT| = |T|$, denote $\parity(k)$ the parity of $k$, and let
\[
l := \begin{cases}
	k+1 & \text{if } \parity(d)\neq\parity(k)\\
	k & \text{if } \parity(d)=\parity(k),
\end{cases}
\
m_j := \begin{cases}
	|I_j| & \text{ if } j \in [l] \text{ and } \parity(d)\neq\parity(k) \\
	|I_{j+1}| &\text{ if } j \in [l-1] \text{ and } \parity(d)=\parity(k) \\
	|I_{k+1}|+|I_1| &\text{ if } j=l \text{ and } \parity(d)=\parity(k).
\end{cases}
\]
If $l=0$, we set $\cD = \emptyset$, and fix a base point $P_{0,1} \in \cT$ to define
$ \tau(t_{1,i}) = P_{0,i} $
for $i \in [n]$. 
Let $l\geq1$ and fix a base point $P_{1,1} \in \cT$.
Iteratively, for each $j=1,\dots,l$ we define $\cI_j = \{P_{j,1},\dots,P_{j,m_j}\}$ such that for each $i \in [m_j -1]$ the point $P_{j,i+1}$ follows consecutively after $P_{j,i}$, and for $j<l$, $P_{j+1,1}$ follows consecutively after $P_{j,m_j}$. 
Therefore,
	\begin{gather} \label{labelling}
		\mathring{T} = \cup_{j=1}^l \cI_j, \quad \cI_j = \{ P_{j,1}, \ldots, P_{j,m_j} \}
	\end{gather}
(see \Cref{fig:circle-ordering}), and it follows that $P_{1,1}$ follows consecutively after $P_{l,m_l}$. 
We define the set of dividers
\[
	\cD = \{ \{P_{j,m_j}, P_{j+1,1} \} \mid j \in [l-1] \} \cup \{ \{P_{l,m_l}, P_{1,1}\} \},
\]
and the map $\tau: T \to \cT$ by
\[
\tau(t_{j,i}) = \begin{cases}
	P_{j,i} &  \text{ if } j \in [k], i \in [n_j] \text{ and } \parity(d)\neq\parity(k) \\
	P_{j-1,i} &  \text{ if } j \in \{2,\dots,k+1\}, i \in [n_j] \text{ and } \parity(d)=\parity(k) \\
	P_{k,m_k+i} &  \text{ if } j =1, i \in [n_1] \text{ and } \parity(d)=\parity(k).
\end{cases}
\]
We obtained a circular composition $(\cT,\cD) = (\tau(T),\cD)$, called the circular composition \emph{induced by} $P_\sigma(T)$ with respect to the base point. When necessary to distinguish between two induced compositions we use the notation $\tau, \tau'$. Note that by construction, $l$ and $d$ have the same parities.

\begin{thm}\label{th:circular-gale}
	Let $P_\sigma(T)$ be a Veronese $d$-polytope with induced circular composition $(\cT,\cD)$ with respect to choices of a base point and cyclic order.
	Then
	$S \subseteq T$ corresponds to a facet of $P_\sigma(T)$ if and only if $\tau(S)$ satisfies the circular facet conditions \eqref{circgale:dividers}, \eqref{circgale:pairs}.
\end{thm}

\begin{proof}
	Let $T=\bigcup_{j=1}^{k+1} I_j$ with $n_j = |I_j|$, as in \Cref{lem:facet-roots}. 
	Recall from \Cref{lem:facet-roots} that $S \subset T$ corresponds to a facet if and only if $S = S_1 \sqcup S_2 \sqcup S_3$, satisfying \eqref{condition-1}, \eqref{condition-2} and \eqref{condition-3}.
	Note, that $\cD=\emptyset$ if and only if $l=k=0$ and $d,k$ have the same parity.
	
	We first show the desired equivalence for the case where $d$ and $k$ have distinct parities. 
	Then $l=k+1$, and $m_j = n_j, j\in [k+1]$.
	With the induced partition $\cT=\bigcup_{j=1}^l \cI$ from \eqref{labelling}, we have
	$\tau(I_j) = \cI_j$.  
	We show that if $S$ satisfies \eqref{condition-1}--\eqref{condition-3}, then $\tau(S)$ satisfies \eqref{circgale:dividers},\eqref{circgale:pairs}.
	Since every consecutive pair in $T$ is a consecutive pair in $\cT$, \eqref{condition-3} implies that $\tau(S_3)$ satisfies \eqref{circgale:pairs}.
	By \eqref{condition-1}, we have $\tau(S_1) = \{\cS_1^1,\dots,\cS_1^{l-1}\}$ such that $\cS_1^j \in \{P_{j,n_j},P_{j+1,1}\}$ for $j = 1,\dots,l-1$. 
	Since $d,k$ have distinct parities, \eqref{condition-2} implies that $|S_2|$ is odd, and therefore $S_2 = \{s\}$ where $s \in \{ t_{1,1},t_{k+1,n_{k+1}} \}$. We thus have $\tau(s) =: \cS_1^{l} \in \{ \tau(t_{1,1}),\tau(t_{k+1,n_{k+1}}) \} = \{ P_{1,1},P_{l,n_{l}} \}$, so $\tau(S_1 \cup S_2)$ satisfies \eqref{circgale:dividers}. Hence, $\tau(S) = \tau(S_1 \sqcup S_2) \sqcup \tau(S_3)$ satisfies the circular facet condition. 
	Conversely, given a set $\cS = \cS_1 \sqcup \cS_2$ satisfying the circular facet conditions, we can reverse the process by applying $\tau^{-1}$. More concretely, we define $S_1 = \{\tau^{-1}(\cS_1^1),\dots,\tau^{-1}(\cS_1^{l-1})\}, S_2 = \{\tau^{-1}(\cS_1^l)\}, S_3 = \tau^{-1}(\cS_2)$.
	
	We now consider the case where $d$ and $k$ have the same parity, and show that if $S$ satisfies \eqref{condition-1}--\eqref{condition-3}, then $\tau(S)$ satisfies \eqref{circgale:dividers},\eqref{circgale:pairs}. 
	If $l=k>0$, then $S_1 \neq \emptyset$ and $\tau(S_1^j) \in \{P_{j-1,m_{j-1}},P_{j,1}\}$ for $j=2,\dots,k$ and $\tau(S_1^1) \in  \{P_{k,m_{k}},P_{1,1}\}$, so $\tau(S_1)$ satisfies \eqref{circgale:dividers}. If $l=k=0$ then $S_1 = \emptyset$, and $\tau(S_1)$ satisfies the condition trivially. Since $d$ and $k$ have the same parity, we have that $S_2$ is either empty or $S_2 = \{t_{1,1}, t_{k+1,n_{k+1}}\}$. In the latter case holds $\tau(S_2) = \{P_{k,n_k}, P_{k,n_k+1}\}$, which is a consecutive pair in $\cT$. Thus, $\tau(S_2 \cup S_3)$ satisfies \eqref{circgale:pairs}. The reverse direction can again be shown by applying $\tau^{-1}$.
\end{proof}

\begin{corollary}\label{cor:isomorphims}
	The polytopes $\polyc(T)$ and $P_{\sigma'}(T')$ are combinatorially equivalent if and only if their induced
	circular compositions $(\tau(T),\cD)$ and $(\tau'(T'),\cD')$ are isomorphic.
\end{corollary}
\begin{proof}
	Let  $(\tau(T),\cD)$ and $(\tau'(T'),\cD')$ be the induced circular decompositions of $\polyc(T)$ and $P_{\sigma'}(T')$.
By definition of isomorphism of circular compositions and maps $\tau, \tau'$, one finds a map $\kappa: T \to T'$ which induces a bijection between subsets of $T$ and $T'$ corresponding to facets. The map $\kappa$ restricts to a map between vertices, and hence induces an isomorphism of face lattices $\polyc(T)$ and $P_{\sigma'}(T')$.

To show the other implication, note that if $V \subset T$ corresponds to the set of vertices of $P_\sigma(T)$, then the induced circular compositions of $\polyc(T)$ and $\polyc(V)$ are isomorphic via the inclusion $V \hookrightarrow T$.
Hence, to prove the claim, it is enough to show it for $\polyc(V)$ and  $P_{\sigma'}(V')$, where $V \subset T$ and $V' \subset T'$ are sets of vertices.
The combinatorial equivalence provides a bijection on vertices which induces a bijection on subsets corresponding to facets. Tautologically, this yields the isomorphism of induces circular compositions.
\end{proof}

\begin{corollary}\label{cor:more-isomorphisms}
Any two induced compositions of a Veronese polytope are isomorphic.
\end{corollary}

	\begin{figure}
	\centering
	\begin{tikzpicture}[scale=1]
	\node at (-4,0.2) {
		\begin{tikzpicture}[scale=1]
			\def \n {9}
			\def \radius {1.5}
			\draw[gray] circle(\radius);
			\draw
			foreach\s in{1,...,\n}{
				(-360/\n*\s:-\radius) node[fill,circle, inner sep=1.5pt]{}
				node[anchor=-360/\n*\s]{$$}
			};
			\draw[thick] (-360/\n*1.5:-\radius-0.4) -- (-360/\n*1.5:-\radius+0.4);
			\draw[thick] (-360/\n*8.5:-\radius-0.4) -- (-360/\n*8.5:-\radius+0.4);
			\node at (-2.2,0) {$\tau(t_{1,1})$};
			\node at (-1.8,1) {$\tau(t_{1,2})$};
		\end{tikzpicture}
	};
	\node at (3.6,0) {
	\begin{tikzpicture}[scale=1]
		\def \n {9}
		\def \radius {1.5}
		\draw[gray] circle(\radius);
		\draw
		foreach\s in{1,...,\n}{
			(-360/\n*\s:-\radius) node[fill,circle, inner sep=1.5pt]{}
			node[anchor=-360/\n*\s]{$$}
		};
		\draw[thick] (-360/\n*11.5:-\radius-0.4) -- (-360/\n*11.5:-\radius+0.4);
		\draw[thick] (-360/\n*13.5:-\radius-0.4) -- (-360/\n*13.5:-\radius+0.4);
		\node at (1.1,-1.7) {$\tau'(t'_{1,2})$};
		\node at (-0.45,-1.9) {$\tau'(t'_{1,1})$};
	\end{tikzpicture}
};
	\node at (-3.5,-3.5) {
		\begin{tikzpicture}
			\draw[gray] (0.4,0) -- (6.3,0);
			\draw[thick] (1.66,0.3) -- (1.66,-0.3)
			foreach \s in{1,...,9}{
				(\s/1.5,0) node[fill,black,circle, inner sep=1.5pt]{}
			};
			\node[anchor=south] at (1/1.5,0) {$t_{1,1}$};
			\node[anchor=south] at (2/1.5,0) {$t_{1,2}$};
		\end{tikzpicture}
	};
	\node at (3.5,-3.5) {
		\begin{tikzpicture}
			\draw[gray] (0.4,0) -- (6.3,0);
			\draw[thick] (2.33,0.3) -- (2.33,-0.3); 
			\draw[thick] (3.66,0.3) -- (3.66,-0.3) 
			foreach \s in{1,...,9}{
				(\s/1.5,0) node[fill,black,circle, inner sep=1.5pt]{}
			};
		\node[anchor=south] at (1/1.5,0) {$t_{1,1}'$};
		\node[anchor=south] at (2/1.5,0) {$t_{1,2}'$};
		\end{tikzpicture}
	};
	\draw[->] (3.55,-3.3) -- (3.55,-2.3);
	\node[anchor=west] at (3.55,-2.8) {$\tau'$};
	\draw[->] (-3.45,-3.3) -- (-3.45,-2.3);
	\node[anchor=east] at (-3.45,-2.8) {$\tau$};
	\draw[->] (-1.5,0.2) -- (1.5,0.2);
	\node[anchor=south] at (0,0.1) {$\sim$};
	\node[anchor=north] at (0,0.2) {$\kappa$};
\end{tikzpicture}
	\caption{Isomorphism between two combinatorially equivalent Veronese polytopes, as described in \Cref{ex:circle-cut}.}
	\label{fig:circle-cut}
\end{figure}
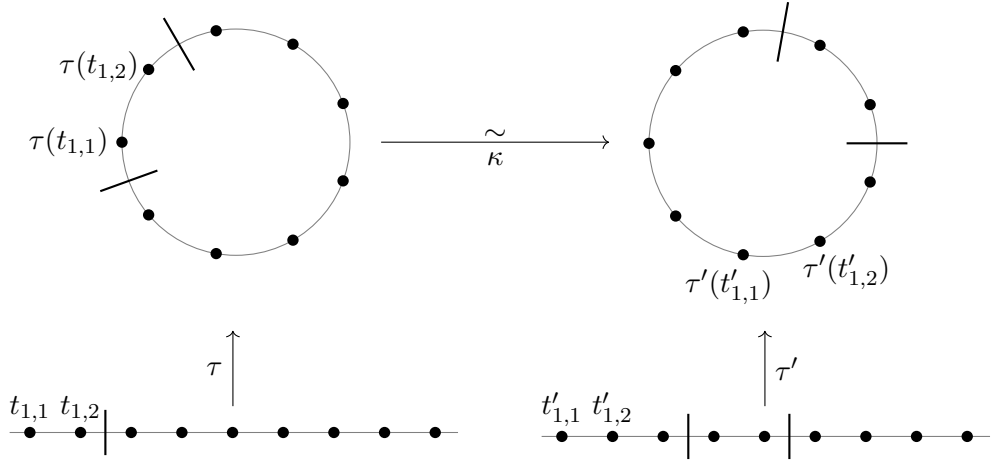

From the above statements, we obtain the following.

\begin{thm}\label{th:circular-compositions-comb-types}
There is a bijective correspondence between combinatorial types of Veronese $d$-polytopes and isomorphism classes of circular compositions with at most $d$ dividers. 
	This bijection is given by the construction of induced circular composition, and maps facets bijectively onto sets satisfying the circular facet condition.
\end{thm}
\begin{proof}
\Cref{th:circular-gale} and 
\Cref{cor:isomorphims,cor:more-isomorphisms}
provide the bijection from combinatorial types of Veronese $d$-polytopes onto their induced circular compositions, which necessarily have at most $d$ dividers.
It remains to
show that the isomorphism class $[(\cT,\cD)]$ of any circular composition with at most $d$ dividers occurs as an isomorphism class of an induced circular composition.
Therefore,
we a fix a representative $(\cT,\cD)$ and a base point $P \in \cT$, and proceed as follows.\par

We induce a linear order $<$ on $\cT$ by declaring $P$ to be the minimal element and for $x,y \in \cT \backslash \{ P \}$ we define $x<y$ if $(P,x,y) \in C$.
Note that points of $(\cT,<)$ can be parametrised by a order-preserving bijective map defined on a subset $T \subset \mathbb{R}$ of cardinality $|T|=|\cT|$ with the standard linear order, on which the dividers induce a partition  $T = \cup_{i=1}^{k+1} I_i$, $k \in \{|\cD|, |\cD|-1\}$.
Equipping the intervals with signs which alter with $i$, we determine a cone $\sigma$ in the hyperplane arrangement $\cup_{t \in T} (\nu_d(t))^0$: there are two such choices of signs, one corresponding to $\sigma$ and the other to $-\sigma$.
The choices of $T$ and $\sigma$ uniquely determine the combinatorial type $P_{\sigma}(T)$, whose induced circular composition is isomorphic with $(\cT,\cD)$ by construction. We thus showed that
the isomorphism class $[(\cT,\cD)]$ is an isomorphism class of an induced circular composition, thereby proving the claim.
\end{proof}

\begin{example}\label{ex:circle-cut}
	Let $d\geq 2$ be even. We consider the $\sigma$-decompositions 
	\begin{align*}
		T &= I_1 \cup I_2 = \{t_{1,1},t_{1,2}\}  \cup \{t_{2,1},\dots,t_{2,7}\}, \text{ and } \\
		T' &= I_1' \cup I_2' \cup I_3' = \{t_{1,1}',t_{1,2}',t_{1,3}'\} \cup \{t_{2,1}',t_{2,2}'\} \cup \{t_{3,1}',\dots,t_{3,4}'\}.
	\end{align*}
	The induced circular compositions $(\tau(T),\cD),(\tau'(T'),\cD')$ are isomorphic, via a composition of ``rotation and reflection'' $\kappa(P_{j,i}) = P'_{j,i}$, where $P_{j,i} = \tau(t_{j,i})$ and  $P'_{j,i} = \tau'(t'_{j,i})$ from the construction of induced compositions (see \Cref{fig:circle-cut}).
\end{example}

\begin{example}\label{ex:cyclic-dividers}
Let $P_\sigma(T)$ be a cyclic polytope realised by the trivial $\sigma$-decomposition $T=I_1$ (compare \Cref{d order curve}), and $C_d(T)$ denote the standard cyclic $d$-polytope generated by the points $T$.
The shape of the induced circular composition of $P_\sigma(T)$ depends on parity of the dimension $d$ as follows.
If $d$ is even, then $l=0$ and hence we have no dividers on the cyclically ordered $\cT$, $|\cT| = |T|$.
Therefore, the isomorphism class of the circular composition $(\cT, \emptyset)$ for $d$ even corresponds to the isomorphism class of $C_d(T)$. 
In the case when $|T|=d+2$, this isomorphism class contains also some circular compositions with two dividers (cf. \Cref{prop:xi-unit-direction-combinatorial}).

If $d$ is odd, then $l=1$ and we have exactly one divider.
Observe that any non-trivial $\sigma$-decomposition $T = I_1 \cup I_2$ in an odd dimension $d$ induces a circular composition with one divider as well.
Since any two circular compositions with exactly one divider are isomorphic in a fixed (odd) dimension, it must be $P_\sigma(T) \cong C_d(T)$.
Therefore, the isomorphism class of circular compositions corresponding $C_d(T)$ contains circular compositions $(\cT,\cD)$ with $|\cT|=|T|$ and $|\cD| = 1$.
\end{example}

\subsection{Combinatorial types}
\label{sec:combinatorial-types}

This section explores combinatorial types of Veronese polytopes, showing that it is a rich class containing many known polytopes.
\Cref{fig:comb-types} exemplifies some of the combinatorial types constructed in this section.

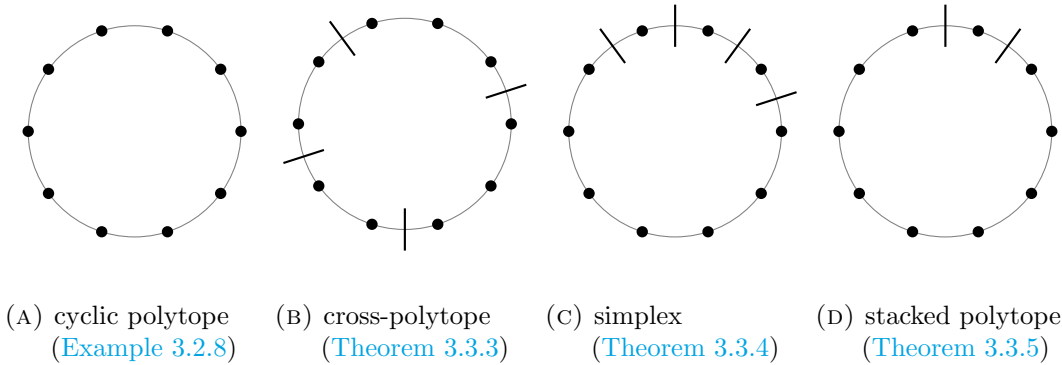
\begin{figure}[h]
	\begin{subfigure}[t]{0.23\textwidth}
		\centering
		\begin{tikzpicture}[scale=0.7]
	\def \n {10}
	\def \radius {2}
	\draw[gray] circle(\radius);
	\draw
	foreach\s in{1,...,\n}{
		(-360/\n*\s:-\radius) node[fill,circle, inner sep=1.5pt]{}
		node[anchor=-360/\n*\s]{$$}
	};
\end{tikzpicture}
		\caption{cyclic polytope \\ 
		\hspace*{1.4em} (\Cref{ex:cyclic-dividers})}
	\end{subfigure}
	\begin{subfigure}[t]{0.23\textwidth}
		\centering
		\begin{tikzpicture}[scale=0.7]
	\def \n {10}
	\def \radius {2}
	\draw[gray] circle(\radius);
	\draw
	foreach\s in{1,...,\n}{
		(-360/\n*\s:-\radius) node[fill,circle, inner sep=1.5pt]{}
		node[anchor=-360/\n*\s]{$$}
	};
	\draw[thick] (-360/\n*1.5:-\radius-0.4) -- (-360/\n*1.5:-\radius+0.4);
	\draw[thick] (-360/\n*4.5:-\radius-0.4) -- (-360/\n*4.5:-\radius+0.4);
	\draw[thick] (-360/\n*7.5:-\radius-0.4) -- (-360/\n*7.5:-\radius+0.4);
	\draw[thick] (-360/\n*9.5:-\radius-0.4) -- (-360/\n*9.5:-\radius+0.4);
\end{tikzpicture}
		\caption{cross-polytope \\ \hspace*{1.4em} (\Cref{th:cross})}
	\end{subfigure}
	\begin{subfigure}[t]{0.23\textwidth}
		\centering
		\begin{tikzpicture}[scale=0.7]
	\def \n {10}
	\def \radius {2}
	\draw[gray] circle(\radius);
	\draw
	foreach\s in{1,...,\n}{
		(-360/\n*\s:-\radius) node[fill,circle, inner sep=1.5pt]{}
		node[anchor=-360/\n*\s]{$$}
	};
	\draw[thick] (-360/\n*1.5:-\radius-0.4) -- (-360/\n*1.5:-\radius+0.4);
	\draw[thick] (-360/\n*2.5:-\radius-0.4) -- (-360/\n*2.5:-\radius+0.4);
	\draw[thick] (-360/\n*3.5:-\radius-0.4) -- (-360/\n*3.5:-\radius+0.4);
	\draw[thick] (-360/\n*4.5:-\radius-0.4) -- (-360/\n*4.5:-\radius+0.4);
\end{tikzpicture}
		\caption{simplex \\
			\hspace*{1.4em} (\Cref{th:simplex})}
	\end{subfigure}
	\begin{subfigure}[t]{0.23\textwidth}
		\centering
		\begin{tikzpicture}[scale=0.7]
	\def \n {10}
	\def \radius {2}
	\draw[gray] circle(\radius);
	\draw
	foreach\s in{1,...,\n}{
		(-360/\n*\s:-\radius) node[fill,circle, inner sep=1.5pt]{}
		node[anchor=-360/\n*\s]{$$}
	};
	\draw[thick] (-360/\n*2.5:-\radius-0.4) -- (-360/\n*2.5:-\radius+0.4);
	\draw[thick] (-360/\n*3.5:-\radius-0.4) -- (-360/\n*3.5:-\radius+0.4);
\end{tikzpicture}
		\caption{stacked polytope \\ \hspace*{1.4em} (\Cref{th:stacked})}
	\end{subfigure}
	\caption{Circular compositions corresponding to $4$-dimensional combinatorial Veronese polytopes on $10$ generating points. (A) and (D) have $10$ vertices, (B) has $8$ vertices and (C) has $5$ vertices.}
	\label{fig:comb-types}
\end{figure}

\begin{thm}\label{thm:comb-types-few-vertices}
	Every $d$-dimensional simplicial polytope on $d+1$, $d+2$ or $d+3$ vertices is combinatorially equivalent to a Veronese polytope.
\end{thm}
\begin{proof}
	Since every simplicial polytope is combinatorially equivalent with a simplicial polytope whose vertices are in general position, we consider, in particular, polytopes $P_1$, $P_2$ and $P_3$ whose $d+1$, $d+2$ and $d+3$ vertices are in a general position, respectively. By embedding the $d$-dimensional affine space containing $P_i$ into a $(d+1)$-dimensional vector space $U$ as an affine chart and erecting a pointed and full-dimensional cone over the polytope we obtain, by projectivising the extremal rays of the cone, $d+1$, $d+2$ or $d+3$ projective points in general position in a $d$-dimensional projective space $\mathbb{P}(U)$. \par 
	
	Recall from \cite[Chapter 1]{harris2013algebraic} that the rational normal curve is uniquely determined by $d+3$ projective points in general position, meaning that one can write a unique map from a projective line $\mathbb{P}(W)$, for some 2-dimensional $W$, into $\mathbb{P}(U)$ which, in the coordinates on $\mathbb{P}(U)$ uniquely determined by this procedure, is given by $[x:y] \mapsto [b_1([x:y] : \cdots : b_{d+1}([x:y]))]$, where $b_1,\ldots,b_{d+1}$ is a basis of homogeneous polynomials of degree $d$ in 2 variables. This allows us to identify $\mathbb{P}(U)$ and $\mathbb{P}(S^dW^*)$, and see such a map $\mathbb{P}(W) \to \mathbb{P}(U)$ as the Veronese embedding $\mathbb{P}(W) \to \mathbb{P}(S^dW^*)$, $\ell \mapsto \ell^0 \otimes \cdots \otimes \ell^0$. \par
	
	Returning to our setting, $d+3$ points uniquely determine the identification, and hence $P_3$ is realised as the affine hyperplane section of a cone compatible with the Veronese factorization structure, i.e., $P_3$ is a Veronese polytope. In case of $d+2$ points, we have a $(d-1)$-dimensional family of rational normal curves passing through these points, showing that there are multiple ways how $P_2$ is combinatorially equivalent with a Veronese polytope. Similarly, in case of $d+1$ points, there are several ways how to realise the simplex $P_1$ as a Veronese polytope.
\end{proof}

	The examples displayed in 
	\Cref{fig:comb-types} show that not all generating points are necessarily vertices. We now give a full characterization of vertices.

\begin{thm}\label{vertices}
Let $\polyc(T)$ be a Veronese polytope with the induced circular composition $(\cT,\cD)$.
If $|\cD|<d$, then all generating points $\cT$ correspond to vertices of $\polyc(T)$. If $|\cD|=d$, then the set of vertices corresponds to the union of its dividers.
\end{thm}

\begin{proof}
	We use the notation as introduced in \Cref{sec:cyclic-gale}.
	Fix
	 $P_{j,i} \in \cT$ for some $j \in [l]$, $i\in [m_j]$, and $\cS = \cS_1 \cup \cS_2 \subset \cT, |\cS|=d$, corresponding to a facet.
	Note that in order to show that $P_{j,i}$ is a vertex, it is enough to show that $P_{j,i}$ belongs to a facet since every facet-supporting hyperplane intersects the factorization curve in at most $d$ points, i.e., the vertices of the facet.
	Thus, if $P_{j,i} \in \cS$ then $P_{j,i}$ is a vertex, so from now on we assume that we assume that $P_{j,i} \not \in \cS$,
	 and construct $\cS' \subset \cT$, $|\cS'|=d$, 
	 corresponding to a facet and containing 
	 $P_{j,i}$.
	We distinguish if $P_{j,i}$ belongs to a divider, or not.
	
	Suppose $P_{j,i}$ is contained in a divider, i.e., $\{P_{j,i},Q\} \in \cD$.
	Then, using \eqref{circgale:dividers} (from \Cref{circular Gale}), we find that $Q \in \cS_1$, and so $Q$ corresponds to a vertex. However, since a vertex cannot be contained in all facets of $\polyc(T)$, there exists a facet $\cS'$ such that $Q \notin \cS'$.
	The condition \eqref{circgale:dividers} implies $P_{j,i} \in \cS'$, and so $P_{j,i}$ corresponds to a vertex.
	This shows that every point contained in a divider is a vertex. In particular, for $|\cD|=d$ the condition \eqref{circgale:dividers} implies that each facet of $\polyc(T)$ has vertices contained only in dividers, thus proving the second claim.

		Suppose $P_{j,i} \in I_j$ does not belong to a divider, therefore $|I_j|\geq 3$, and assume that $|\cD| < d$.
		Then there exists a generating point $Q \in \cI_j$ such that $\{P_{j,i}, Q\}$ is a consecutive pair (in fact, there exist exactly two such points.)  Observe that $|\cD| < d$ implies $\cS_2 \neq \emptyset$, and thus there exists a consecutive pair $\{ a,b \} \in \cS_2$. We have the following cases.
		\begin{enumerate}
			\item If $Q \in \cT \setminus \cS$, then $S' = (\cS \setminus \{a,b\})\cup \{P_{j,i},Q\}$ satisfies \eqref{circgale:dividers} and \eqref{circgale:pairs}, and hence is a facet of $\polyc(T)$ containing $P_{j,i}$. \label{item:vertices-case}
			\item If $Q \in \cS_2$, then \eqref{circgale:pairs} implies that $\{Q,Q'\} \subset \cS_2$ form a consecutive pair for some $Q' \in \cT$, and that $(\cS \setminus \{Q, Q'\}) \cup \{P_{j,i},Q\}$ corresponds to a facet containing $P_{j,i}$.
			\item If $Q \in \cS_1$ such that $\{P_{j,i},Q\}$ form a consecutive pair, then, as before, $Q$ cannot be contained in all facets of $\polyc(T)$, and hence there exists a facet $\cS'$ not containing $Q$. If $P_{j,i} \in \cS'$, then it corresponds to a vertex. Otherwise, $\{Q,P_{j,i}\} \subset \cT\setminus \cS$ form a consecutive pair, which is covered by case \eqref{item:vertices-case}.
		\end{enumerate}
	\vspace*{-1.5em}
\end{proof}

Recall that the $d$-dimensional \emph{cross-polytope} in $\R^d$ is the convex hull of $\pm e_1,\dots,\pm e_d$, where $e_1,\dots,e_d$ is the standard basis of $\R^d$.

\begin{thm}[Cross-polytope]\label{th:cross}
Let $|T|=n$ and $\polyc(T)$ be a Veronese $d$-polytope, $d \geq 3$, such that the signed $\sigma$-decomposition partitions $T$ into $d+1$ intervals, $T = \bigcup_{j=1}^{d+1} I_j$.
If $|I_j| \geq 2$ for all $j \in [d+1]$, then $\poly(T)$ is combinatorially equivalent to the $d$-dimensional cross-polytope.
\end{thm}
\begin{proof}
	The polytope $P_\sigma(T)$ is combinatorially equivalent to the $d$-dimensional cross-polytope if and only if  it is has $2d$ vertices $e_1^+,\dots,e_d^+,e_1^-,\dots,e_d^- \in T$ and a subset $S \subset T$ of cardinality $d$ forms a facet if and only if $|S \cap \{e_j^+, e_j^-\} |=1$ for all $j \in [d]$. 
	Recall from \Cref{lem:facet-roots}, that $S \subset T$ corresponds to a facet of $\polyc(T)$ if and only if $S = S_1 \sqcup S_2 \sqcup S_3$ satisfies \eqref{condition-1}--\eqref{condition-3}, and since $|S_1| = d$, we have $S_2 = S_3 = \emptyset$.
	Finally, if we define $e_j^- := t_{j-1,n_{j-1}}$ and $e_j^+ := t_{j,1}$, $j =2,\dots,d+1$, where $n_j = |I_j|$, then \eqref{condition-1} is equivalent to $\polyc(T)$ being a cross-polytope.
\end{proof}

\begin{thm}[Simplex]\label{th:simplex}
Let $|T|=n$ and $\polyc(T)$ be a Veronese $d$-polytope, $d \geq 3$, such that the signed $\sigma$-decomposition partitions $T$ into $d+1$ intervals, $T = \bigcup_{j=1}^{d+1} I_j$.
If $|I_1| = |I_2| = \dots = |I_{d}| = 1$ and $|I_{d+1}| = n-d$, then $\polyc(T)$ is a $d$-dimensional simplex.
\end{thm}
\begin{proof}
	Let $S \subseteq T$ correspond to a facet. Then by \Cref{lem:facet-roots} we have that $S = S_1$, satisfying \eqref{condition-1}. For $j=1,\dots,d$ we have $t_{j,1} = t_{j,n_{j}}$, so $S \subset \{t_{1,1},t_{2,1},\dots,t_{d,1},t_{d+1,1}\}$. It follows that $\polyc(T)$ has at most $d+1$ vertices. By \Cref{prop:simplicial}, $\polyc(T)$ is full-dimensional, and thus $\polyc(T)$ is a $d$-dimensional simplex.
\end{proof}

We note that the conditions of \Cref{th:simplex} are not the only ones for which $\polyc(T)$ is a simplex. For example, if $d=n-1$ then $\poly(T)$ is a simplex for any $T = \bigcup_{j=1}^{k+1} I_j$. 

In the following, we show that for every $d$ and $n$, there exists a $d$-dimensional Veronese polytope that is a stacked polytope. 
A polytope is \emph{stacked} if it can be built from a $d$-dimensional simplex by a sequence of stackings. To define the stacking operation, let $P = \{x \in \R^d \mid \forall i \in [m]: \ \inner{x,u_i} \geq \alpha_i \}$ be a $d$-polytope with facets $F_i = \{x \in P \mid \inner{x,u_i} = \alpha_i\}$, $i \in [m]$. Fix $j \in [m]$ and $p \in \R^d$ such that $\inner{p,u_i} > \alpha_i$ for all $i \neq j$, and $\inner{p,u_j} < \alpha_j$. Then a \emph{stacking} onto the facet $F_j$ is the polytope $\conv(p,P)$, whose combinatorial type only depends on the choice of $j \in [m]$. Indeed, the facets of $\conv(P,p)$ are 
\begin{enumerate}[label={\arabic*)}]
	\item $F_i, i \in [m] \setminus \{j\}$,
	\item pyramids $\conv(G,p)$, where $G$ ranges over all $(d-2)$-dimensional facets of $F_j$.
\end{enumerate}
It is known that in dimensions $d\geq 4$, stacked polytopes are precisely the minimizers of the lower bound theorem, minimizing the number of facets of a simplicial polytope.

\begin{thm}[Stacked polytope]\label{th:stacked}
Let $|T|=n$ and $\polyc(T)$ be a Veronese $d$-polytope, $d \geq 3$, such that the signed $\sigma$-decomposition partitions $T$ into $d-2$ intervals, $T = \bigcup_{j=1}^{d-2} I_j$.
If $|I_1| = |I_2| = \dots = |I_{d-3}| = 1$ and $|I_{d-2}| = n-(d-3)$, then $\polyc(T)$ is a $d$-dimensional stacked polytope on $n$ vertices.
\end{thm}
\begin{proof}
	With indexing $T = \{t_1,\dots,t_n\}$ let $p = t_{n-1}$, $\bar{S} = \{t_1,\dots,t_{d-2},t_{n-2},t_n\} \subset T$ and
	fix any $\xi \in \interior(\sigma)$. 
	We denote $P_n = \poly(T)$ and $P_{n-1} = \poly(T\setminus \{p\})$, and show that $P_n$ is a stacking over $P_{n-1}$ onto the facet corresponding to $\bar{S},$ with additional vertex corresponding to $p$.
	Since $P_{n-1}$ has signed $\sigma$-decomposition $T \setminus \{p\} = \bigcup_{j=1}^{d-2} I'_j$ so that
	$|I'_1| = |I'_2| = \dots = |I'_{d-3}| = 1$ and $|I'_{d-2}| = n-(d-3)-1$, the following argument can then be applied iteratively, terminating at the simplex $P_{d+1}$, thereby proving the claim. \medskip
	
	To show that $P_n$ is a stacking over $P_{n-1}$,
	we need to prove that 
	\begin{enumerate}[label={\arabic*)}]
		\item for all $S \subset T \setminus \{p\}, S \neq \overline{S}$, the set $S$ corresponds to a facet of $P_n$ if and only if $S$ corresponds to a facet of $P_{n-1}$, and that
		\item for all $S \subset T, p \in S$, the set $S$ corresponds to facet of $P_{n}$ if and only of $S \setminus \{p\}$ is a $(d-2)$-dimensional facet of $\overline{S}$. 
	\end{enumerate}

	For the first part, recall from \Cref{geometric Gale} that a set $S \subset T$ is a facet of $P_n$ if and only if $|S|=d$ and $\lambda_{S,\xi} = \frac{p_S}{q_\xi}$ has constant sign on $t \in T \setminus S$.
	In particular, for any facet $S \subset T \setminus \{p\}$ of $P_n$ holds that $S \neq \bar{S}$, and moreover
	$\lambda_{S,\xi}$ has constant sign on all $(T \setminus \{p\}) \setminus S$, so $S$ corresponds to a facet of $P_{n-1}$. 
	
	To prove the other direction, let $S \subset (T \setminus \{p\})$ correspond to a facet of $P_{n-1}$. 
	If $t_{n-2} \not \in S$, then $\sgn(\lambda_{S,\xi}(t_{n-2})) = \sgn(\lambda_{S,\xi}(t_{n-1})) \neq 0$. Thus, $\lambda_{S,\xi}$ has constant sign on all $T\setminus S$, so $S$ is  facet of $P_n$. A similar argument applies in the case where $t_n \not \in S$.
	It remains to show the statement if $t_{n-2},t_n \in S$. 
	Following \Cref{cor:bijection-facets-xi-pa} or \Cref{lem:facet-roots}, it can be checked that the facets of $P_{n-1}$ containing both $t_{n-2}$ and $t_n$ are either 
	\[
	S = \overline{S} \quad \text{ or } \quad S = S_i = \{t_j \mid j \in [d-2] \setminus \{i\}\} \cup \{t_{n-3},t_{n-2},t_n\} \text{ for } i \in [d-2] .
	\]
	Since $S_i$ is a facet of $P_{n-1}$, it follows that $\lambda_{S_i,\xi}$ has constant sign on $(T\setminus \{p\}) \setminus S_i$.
	In order to show that $S_i$ is a facet of $P_n$, it remains to show that 
	$\sgn\!\left(\frac{p_{S_i}(t_{i})}{q_\xi(t_{i})}\right) = \sgn\!\left(\frac{p_{S_i}(t_{n-1})}{q_\xi(t_{n-1})}\right)$, or, equivalently, that
	\[
	\sgn(p_{S_i}(t_i)) \sgn(q_\xi(t_i)) = 				\sgn(p_{S_i}(t_{n-1})) \sgn(q_\xi(t_{n-1})) .
	\]
	Since $\{t_i\} = I_i$, $t_{n-1} \in I_{d-2}$ and the sign of $q_\xi$ is alternating on intervals $I_j$, it follows that
	$\sgn(q_\xi(t_{n-1})) = (-1)^{d-i-2} \sgn(q_\xi(t_{i})) $. The number of roots of $p_{S_i}$ between $t_i$ and $t_{n-1}$ equals $d - (i-1) - 1$, so $\sgn(p_S(t_{n-1})) = (-1)^{d-i} \sgn(p_S(t_{i})) $. Hence,
	\[
	\sgn(p_S(t_{n-1})) \sgn(q(t_{n-1})) =
	(-1)^{2d - 2i - 2}
	\sgn(p_S(t_i)) \sgn(q(t_i)) = 	\sgn(p_S(t_i)) \sgn(q(t_i))  ,
	\]
	so $S_i$ is a facet of $P_n$. 
	
	For the second part, following again \Cref{cor:bijection-facets-xi-pa} or \Cref{lem:facet-roots}, it can be checked that the facets of $P_{n}$ containing $p = t_{n-1}$ are 
	$S = \overline{S} \setminus \{t_j\} \cup \{p\}$ for $t_j\in \overline{S}$. 
	Because $P_{n-1}$ is a simplicial polytope, $\overline{S}$ is a simplex, hence the facets of $\overline{S}$ are $\overline{S} \setminus \{t_j\}, t_j \in \overline{S}.$ 
	It follows that $S$ is a facet of $P_{n}$ containing $p$ if and only if $S \setminus \{p\}$ is a facet of $\overline{S}$.
\end{proof}

Recall from \Cref{ex:cyclic-dividers} that odd-dimensional Veronese polytopes with one divider correspond to cyclic polytopes. We now turn our attention to this setting for even dimensions.
A polytope $P$ is \emph{$k$-neighbourly} if any subset of $k$ or less vertices form a face of $P$. A $\lfloor d/2 \rfloor$-neighbourly polytope is called \emph{neighbourly}. 

\begin{proposition}\label{prop:xi-unit-direction-combinatorial}
Let $\polyc(T)$ be a Veronese polytope of even dimension $d \geq 4$, such that the $\sigma$-decomposition partitions $T$ into $2$ intervals, $T = \{t_1,\dots,t_n\} = I_1 \cup I_2$,  $|T|=n$.
Then the following holds.
	\begin{enumerate}[label={\textup{(}\roman*\textup{)}},ref={\roman*}]
		\item If $n = d+2$ and $|I_1|,|I_2|$ are both even, then $\polyc(T)$ is cyclic. \label{item:case4a}
		\item \label{item:case-new} If $n = d+2$ and $|I_1|,|I_2|$ are both odd, then $\polyc(T)$ is not a neighbourly polytope. 
		\item If $n > d+2$ then $\polyc(T)$ is not a neighbourly polytope.  \label{item:case4}
	\end{enumerate}
	In particular, the polytopes in \eqref{item:case-new} and \eqref{item:case4} are not cyclic.
\end{proposition}
An example of \eqref{item:case4} is provided in \Cref{ex:facets}.

\begin{proof}
	\eqref{item:case4a}.
	Let $m = |I_1|$ and $\pi \in S_n$ be the permutation given by $\pi(i) = i$ for $i = 1,\dots,m$ and $\pi(i) = n+m+1-i$ for $i = m+1,\dots,n$.
	By \Cref{cor:bijection-facets-xi-pa}, facets of $\polyc(T)$ correspond to $\sigma$-parity alternating sequences of length $n-d = 2$. Because of this short length, it is easy to verify that $(l_1,l_2)$ is $\sigma$-p.a.\ if and only if $(\pi(l_1),\pi(l_2))$ is p.a.\, which proves that $\pi$ induces a bijection between facets of $\polyc(T)$ and facets of the cyclic polytope $C_d(n)$.

	\eqref{item:case-new} and \eqref{item:case4}. 
	To show that $\polyc(T)$ is not neighbourly, it is enough to find
	a set $V$ of cardinality $d/2$ corresponding to vertices of $\polyc(T)$ which are not contained in any facet.
	To this end, we define $V = V_1 \cup V_2$, for
	\[
		V_1 = \{m-2k_1 +1,\dots,m-3,m-1\} \quad \text{and} \quad V_2 = \{m+2,m+4,\dots,m+2k_2\},
	\]
	where 
 	$k_1 + k_2 = \tfrac{d}{2}$ are such that $m \geq 2k_1 +1$ and  $n-m \geq 2k_2 + 1$. The assumptions of \eqref{item:case-new}, \eqref{item:case4} ensure that such $k_1$ and $k_2$ exist.
 	We verify that 
 	$\{t_j \mid j \in V\}$ is not contained in any facet by showing that $[n] \setminus V$ does not contain a $\sigma$-p.a.\ sequence.

	To prove the claim by contradiction, suppose that $L = (l_1,\dots,l_{n-d}) \subset [n] \setminus V$ is $\sigma$-p.a.  Restricted to $[m]$, $L$ is parity-alternating, and thus 
	\[L \cap [m] \subseteq \{1,2,\dots,m-2k_1 - 1\} \cup \{m-2k_1 +2i\} \quad \text{ for some } i \in \{0,\dots,k_1\} .
	\]
	Similarly, $L$ is p.a.\ when restricted to $\{m+1,\dots,n\}$, implying that
	\[
	 L \cap \{m+1,\dots,n\} \subseteq \{m+2j + 1\}\cup \{m+2k_2+2,m+2k_2+3,\dots,n-1,n\}.
	 \]
	  for some $j \in \{0,\dots,k_2\}$.
	We obtain that
	\[
		L \subseteq \tilde{L}_{ij} = [m-2k_1 -1] \cup \{m-2k_1 + 2i, m+2j + 1\} \cup \{m+2k_2+2,m+2k_2+3,\dots,n\},
	\]
	where $|\tilde{L}_{ij}| = n-d$. As $|L|=n-d$ by assumptions, we deduce that $L = \tilde{L}_{ij}$ for some $i \in \{0,\dots,k_1\}, j \in \{0,
	\dots,k_2\}$. However, the parities of $m-2k_1 + 2i, m+2j + 1$ are distinct, thereby contradicting that $L$ is $\sigma$-p.a.
\end{proof}

\begin{corollary}\label{cor:3d}
	Any Veronese $3$-polytope is a cyclic polytope or an octahedron.
\end{corollary}
\begin{proof}
	We use the characterization from \Cref{cor:isomorphims} to describe the isomorphism class of $\polyc(T)$ via the isomorphism class of the induced circular composition $(\cT,\cD)$.
	Since $d=3$ is odd, we have that $|\cD| \in \{1,3\}$.
	For $|\cD|=1$, \Cref{ex:cyclic-dividers} implies that $\polyc(T)$ is a cyclic polytope on $n$ vertices. If $|\cD|=3$, then $\cT = \cI_1 \cup \cI_2 \cup \cI$.
	If $|I_j|\geq 2$ for all $j\in [3]$, then $\polyc(T)$ is an octahedron, the $3$-dimensional cross-polytope, by \Cref{th:cross}. 
	For the case when exactly one of $I_1,I_2,I_3$ has cardinality $1$, we can assume without loss of generality that $|\cI_1|=1, |\cI_2|=a, |\cI_3|=b$, where $|T|=1+a+b$ and $a,b \geq 2$. 
	It follows from \eqref{circgale:dividers} that $\polyc(T)$ has only $5$ vertices, namely the union of all dividers.
	Since the cyclic polytope $C_3(5)$ is the unique combinatorial type of simplicial $3$-polytopes with $5$ vertices, we found $\polyc(T) \cong C_3(5)$.
 	Finally, if exactly two of the sets $I_1,I_2,I_3$ have cardinality $1$, then $\polyc(T)$ is a $3$-dimensional simplex by \Cref{th:simplex}, being a cyclic polytope on $4$ vertices.
\end{proof}

\begin{rem}[Computational results]\label{rem:computational}
	\Cref{table:small-dimensions} (on \cpageref{table:small-dimensions}) and \Cref{table:more-comb-types} (on \cpageref{table:more-comb-types}) show the numbers of distinct combinatorial types of Veronese polytopes of dimension $d \leq 7$ with $n$ vertices, for small $n$. These results were obtained by a computational enumeration
	via \texttt{SageMath} \cite{sagemath}. 
	In each of these computations, we additionally checked how many of the obtained combinatorial types are neighbourly, and how many types are stacked. 
	For every pair $(d,n), d \geq 4$ from \Cref{table:small-dimensions} or \labelcref{table:more-comb-types} with well-defined values, 
	we found that
	there is exactly one combinatorial type that is a stacked polytope, and that for most pairs $(d,n)$, the cyclic polytope $C_d(n)$ is the only neighbourly polytope (with the exception of $(5,8)$ with two neighbourly types, and $(7,10)$ with $4$ neighbourly types). Together with \Cref{thm:comb-types-few-vertices}, this indicates that for any $(d,n)$ such that $n > d+3$, cyclic polytopes might be the only neighbourly Veronese polytopes, and the polytope as constructed in \Cref{th:stacked} might be the only stacked Veronese polytope.
\end{rem}

\subsection{Number of facets of Veronese polytopes}

In this section, we count the number of facets of Veronese polytopes.
In the following, for fixed $l\in \N$ and $r$ a non-negative integer
we denote 
\begin{align*}
	\Delta_{r}= \{(r_1,\dots,r_{l}) \in \Z^l \mid r_j \geq 0, \sum_{j=1}^{l} r_j = r\} ,
\end{align*}
and use the conventions that $\prod_{i \in \emptyset} x_i = 1$, and that for the binomial coefficient holds $\binom{n}{k} = 0$ whenever $n<0$, and $\binom{n}{0}=1$ for $n\geq 0$. Given two sets $A = \{a_1,\dots,a_{q}\}, M = \{m_1,\dots,m_{q}\} \subset [l]$, we say that $A$ and $M$ are \emph{interlacing} if $a_1 < m_1 < a_2 < m_2 < \dots < m_q$ or $m_1 < a_1 < m_2 < a_2 < \dots < a_q$. 

\begin{thm}[Number of facets of Veronese polytopes]\label{th:facet-count} Let $\polyc(T)$ be a $d$-dimensional Veronese polytope with induced circular composition $(\cT,\cD)$ such that $\cD \neq \emptyset$,
	 $\cT = \bigcup_{j=1}^l \cI_j$ and $m_j = |\cI_j|$. The number of facets of $\polyc(T)$ is
		
		\begin{equation*}
			\begin{split}
				\sum_{(r_1,\dots,r_{l}) \in \Delta_{\frac{d-l}{2}}} \ 
				\sum_{\substack{A, B \subset [l] \\ A,B \text{ interlacing}   \\ 1 \leq |B|\leq \lfloor l/2 \rfloor } } \ 
				\prod_{j \in A} \sma m_j  - r_j \\ r_j\strix  
				&\prod_{j \in B} \sma m_j - 2- r_j \\ r_j \strix 
				\prod_{j \in [l] \setminus (A \cup B)} \sma m_j - 1 - r_j \\ r_j \strix \\
				&+2 \sum_{(r_1,\dots,r_{l}) \in \Delta_{\frac{d-l}{2}}} \ \prod_{j \in [l]} \sma m_j - 1 - r_j \\ r_j \strix  \ .
			\end{split}
		\end{equation*}

\end{thm}

Recall that if $\cD = \emptyset$, then the corresponding polytope is a cyclic polytope in even dimension, and the number of facets is $\binom{n-d/2}{d/2} + \binom{n - 1  - d/2}{d/2 -1}$, which is the number of consecutive pairs in (cyclically ordered) $\cT$, $|\cT| = n$.

\begin{proof}
	Recall that $\cS \subset \cT$ defines a facet of $\polyc(T)$ if and only if $\cS = \cS_1 \cup \cS_2$, satisfying conditions \eqref{circgale:dividers}, \eqref{circgale:pairs} from \Cref{circular Gale}. 
	Condition \eqref{circgale:dividers} describes how to choose one point for each divider, forming the set $\cS_1$. By construction, we have $|\cS_1 \cap \cI_j| \in \{0,1,2\}$ for each $j\in[l]$.
	With
	\[
		D_i = |\{j \in [l] \mid |\cI_j \cap \cS_1| = i\}|
	\]
	we find
	\begin{gather*}
		2\cdot D_2 +1\cdot D_1 + 0\cdot D_0 =l \quad \text{and} \quad
		D_2 + D_1 + D_0 = l,
	\end{gather*}
	where the first equation expresses that $|\cS_1| = l$, and the second that there are $l$ intervals $\cI_1,\dots,\cI_l$. These two equations imply that $D_0 = D_2$, and 
	since $D_1 \geq 0$,
	the first equation gives the bound $D_2 \leq \lfloor l/2 \rfloor$.

	We parametrise choices of $l$ points for $\cS_1$. Let $B =\{b_1,\ldots,b_{D_2}\} \subset [l]$ label the intervals with 2 elements in $\cS_1$, and let $A = \{a_1,\ldots,a_{D_0}\} \subset [l]$ label intervals with $0$ elements in $\cS_1$. Then, the condition \eqref{circgale:dividers} implies that 
	if $A,B \neq \emptyset$ then $A$ and $B$ are interlacing.
	Moreover, \eqref{circgale:dividers} implies that $A$ and $B$ uniquely determine $\cS_1$.
	Indeed, if $b_i,b_{i+1} \in B, a \in A$ are such that $b_i < a < b_{i+1}$, then for each $j \in [l]$ such that $b_i < j < a$ we have $t_{j,n_j} \in \cS_1$, and for $a < j < b_{i+1}$ we have $t_{j,1} \in \cS_1$.
	One similarly determines the values for $j<b_1$ and $j>b_{|D_2}$.

	Let $A$ and $B$ be as above and such that $D_0 = D_2 \neq 0$. This fixes the choice of $\cS_1$. We now count the number of choices of $d-l$ pairs in accordance with \eqref{circgale:pairs}. First, observe that there are ${m-r \choose r}$ choices of $r$ distinct consecutive pairs in $\{1,\ldots,m\}$.
	Let $r$ be such that $d-l=2r$.
	To complete $\cS_1$ into a facet $\cS = \cS_1 \cup \cS_2$ we are left with choosing $r$ pairs, which we do by choosing $r_j$ pairs in $I_j$, $j \in [l]$, such that $\sum_{j=1}^{l} r_j=r$. The number of such $\cS_2$ is thus
	\begin{equation*}
		\sum_{(r_1,\dots,r_{l}) \in \Delta_{r}} \prod_{j \in A} \binom{m_j  - r_j}{r_j} \prod_{j \in B} \binom{m_j -2- r_j}{r_j} \prod_{j \in [l] \setminus (A \cup B)} \binom{m_j - 1 - r_j}{r_j}.
	\end{equation*}

	We now consider the case where $D_0=D_2=0$.
	Then $\cS_1 = \{t_{j,1} \mid j \in [l]\}$ or $\cS_1 = \{t_{j,n_j} \mid j \in [l]\}$. These two cases together yield
		\begin{equation*}
		2\sum_{(r_1,\dots,r_{l}) \in \Delta_{r}} \prod_{j \in [l]} \binom{m_j - 1 - r_j}{r_j} \ .
	\end{equation*}
	Summing over all choices of $A$ and $B$ yields the final result.
\end{proof}

\begin{example}[Facets of odd-dimensional cyclic polytopes]\label{ex:facets-cyclic}
	We recover the number of facets of odd-dimensional cyclic polytopes from \Cref{th:facet-count}. 
	Therefore, let $d$ be odd and $l=1$ (see \Cref{ex:cyclic-dividers}).
	Then, $\lfloor \frac{l}{2} \rfloor= 0$, and the first sum is void.
	The indexing set of the second sum consists of a single point $r_1 = \frac{d-1}{2}$, and with $n=m_1$ we find
	\[
		2 \binom {n -1 - \frac{d-1}{2}}{ \frac{d-1}{2}}.
	\]
\end{example}

\begin{example}
	\label{ex:not-neigborly-faets}
	We consider Veronese polytopes with two dividers in four dimensions. 
	Concretely, we choose $T = \{-3,-2,-1,1,2,3,4\}$ and $\xi = (0,-1,0,0,0)$ as in \Cref{ex:facets}, so $m_1 = 3$ and $m_2 = 4$. 
	The number of facets if $\poly(T)$ is
	\begin{multline*}
		\sum_{(r_1,r_{2}) \in \Delta_{1}} 
		\sum_{\substack{A\subset [2] \\ |A| = 1 }} \ 
		\prod_{j \in A} \sma m_j - r_j \\ r_j \strix \prod_{j \in B = [2] \setminus A} \sma m_j - 2 - r_j \\ r_j \strix 
		+
		2 \! \! \! \sum_{(r_1,r_{2}) \in \Delta_{1}} \prod_{j \in [2]} \sma m_j - 1 - r_j \\ r_j \strix
		\\
		= \binom{m_1 - 3}{1} + \binom{m_1 - 1}{1} + \binom{m_2 - 3}{1} + \binom{m_2 - 1}{1} + 2\left(\binom{m_1 - 2}{1}+\binom{m_2 - 2}{1}\right) = 12.
	\end{multline*}	
	Note that a $4$-dimensional neighbourly polytope (just like the cyclic polytope) with $7$ vertices has $14$ facets, so $\poly(T)$ is not neighbourly.
\end{example}

\vspace{2em}

\begin{table}[h]
	\begin{tabular}{|c|c|c|c|}
		\hline
		$n$ & $d=4$ & $d=5$ & $d=6$ \\ \hline
		5 &         1  &             &            \\
		6 &         2 &          1  &            \\
		7 &         5 &          2  &    1      \\
		8 &         6 &          8 &     3     \\
		9 &         5 &          9 &   18     \\
		10 &         6 &       10 &     24    \\
		11 &         6 &        11 &     27   \\
		12 &         7 &        13 &    37   \\
		13 &         7 &       15 &      42  \\
		14 &         8 &       17 &     55   \\
		15 &         8 &       20 &       ?     \\
		16 &         9 &       22 &       ?     \\
		17 &         9 &       25 &        ?    \\
		18 &        10 &      28 &         ?   \\
		19 &        10 &       31 &        ?    \\
		20 &        11 &         ?  &        ?    \\
		21 &        11 &         ?  &         ?   \\ \hline
	\end{tabular}
	\vspace{.5em}
	\caption{Computational results for the number of combinatorial types of Veronese polytopes per dimension $d = 4,5,6$ and number of vertices $n$. This extends the results displayed in \Cref{table:small-dimensions}.}
	\label{table:more-comb-types}
\end{table}


\printbibliography

\end{document}